\documentclass[11pt,letterpaper]{amsart}
\usepackage{amssymb}
\usepackage{mathrsfs}
\usepackage[pagebackref, colorlinks = true,
linkcolor = red,
urlcolor  = blue,
citecolor = blue,
anchorcolor = blue]{hyperref}
\hypersetup{pdftitle={title},pdfauthor={author}}

\theoremstyle{plain}

\newtheorem{theorem}{Theorem}[section]
\newtheorem{mainthm}{Theorem}

\newtheorem{corollary}[theorem]{Corollary}
\newtheorem{lemma}[theorem]{Lemma}

\newtheorem{remark}[theorem]{Remark}
 \newcommand{\sg}{\ensuremath{\sigma}}
 \newcommand{\ep}{\ensuremath{\varepsilon}}
 \newcommand{\om}{\omega}
\newcommand{\zt}{\ensuremath{\zeta}}
\DeclareMathOperator{\lin}{lin}
\DeclareMathOperator{\Exp}{Exp}

\newcommand{\sF}{\ensuremath{\mathscr{F}}}
\newcommand{\EE}{\ensuremath{\mathbb E}}
\newcommand{\NN}{\ensuremath{\mathbb N}}%Natural Numbers
\newcommand{\PP}{\ensuremath{\mathbb P}}
\newcommand{\RR}{\ensuremath{\mathbb R}}% Real Numbers
\newcommand{\ZZ}{\ensuremath{\mathbb Z}} % Gives you a shortcut for writing the blackboard Z for the integer numbers - \ZZ

\allowdisplaybreaks

\numberwithin{equation}{section}
\title[Universal Gap Growth]{Universal Gap Growth for Lyapunov Exponents of Perturbed Matrix Products}
\date{\today}

\author[Atnip]{Jason Atnip}
\address{School of Mathematics and Statistics, University of New South Wales, Sydney, NSW 2052, Australia}
\email{\href{j.atnip@unsw.edu.au}{j.atnip@unsw.edu.au} }

\author[Froyland]{Gary Froyland}
\address{School of Mathematics and Statistics, University of New South Wales, Sydney, NSW 2052, Australia}
\email{\href{g.froyland@unsw.edu.au}{g.froyland@unsw.edu.au} }

\author[Gonz\'alez-Tokman]{Cecilia Gonz\'alez-Tokman}
\address{School of Mathematics and Physics, The University of Queensland, Brisbane, QLD 4072, Australia}
\email{\href{cecilia.gt@uq.edu.au}{cecilia.gt@uq.edu.au} }

\author[Quas]{Anthony Quas}
\address{Department of Mathematics \& Statistics, University of Victoria, Victoria, BC, Canada V8W 2Y2}
\email{\href{aquas@uvic.ca}{aquas@uvic.ca} }

\newcommand{\reqgap}{\gap}%{\ensuremath{\varepsilon^{-\rgapindex}}}
\newcommand{\gap}{\ensuremath{\mathscr{G}}}
\newcommand{\blocklength}{\ensuremath{N}}
\newcommand{\target}[1]{#1^{\odot}}
\newcommand{\mr}[1]{#1}%{#1^\circ}
\newcommand{\tarpert}[1]{#1^{\times}}
\newcommand{\tinygap}{\ensuremath{(\tfrac{\varepsilon}{4})^{N}}/(3dN)}%{\ensuremath{\delta}}
\begin{document}

%\tableofcontents

\begin{abstract}
    We study the quantitative simplicity of the 
    Lyapunov spectrum of $d$-dimensional bounded matrix cocycles
    %a sequence of matrices of uniformly bounded norm 
    subjected to additive random perturbations. In dimensions 2 and 3, 
    we establish explicit lower bounds on the gaps between consecutive
    Lyapunov exponents of the perturbed cocycle, depending only on the  
    %maximum 
    scale of the perturbation. 
    In arbitrary dimensions, we show existence of a universal lower bound on these gaps.
    % \aq{I changed ``maximum amplitude" to ``scale" as the results apply to
    % normal perturbations (for example) also.}\cgt{looks good}
    % %, and uniform over all choices of the original sequence.
    A novelty of this work is that the bounds provided are uniform over all choices of the original sequence of matrices. 
    Furthermore, we make no stationarity assumptions on this sequence. Hence, our results apply to random and sequential dynamical systems alike.
\end{abstract}

\maketitle

\section{Introduction}
We study the Lyapunov spectrum of sequences of uniformly norm-bounded matrices with additive noise-like perturbations:
we start with a sequence of matrices, which we assume, without loss of generality, to be of norm at most 1. 
The entries of each matrix are independently perturbed by adding a random number
uniform in the range $[-\ep,+\ep]$.
The unperturbed matrices are denoted $(A_i)_{i\in\NN}$, and we write $A_{i,\ep}$ to denote the
perturbed matrix $A_i+\ep\Xi_i$ where $\Xi_i$ are independent matrix random variables with independent
entries uniformly distributed on $[-1,1]$. 
We are interested in universal quantitative simplicity of the Lyapunov spectrum. That is, 
we establish lower bounds on the exponential growth rate of the ratio of the $j$th and
$(j+1)$st singular values of the product $A_\ep^{(n)}=A_{n,\ep}\cdots A_{1,\ep}$.
Our bounds depend only on the maximum amplitude  $\ep$ of the 
noise, and not on the particular sequence of matrices. 

Our results are in part motivated by questions concerning the behaviour of Lyapunov exponents of differentiable dynamical systems in the presence of noise. If Lyapunov exponents are estimated using measurements of the derivative cocycle, it is natural to model the measurement error as a random perturbation of the underlying derivative cocycle. We focus on the abstracted question: given an arbitrary bounded sequence of matrices, subjected to noise-like perturbations, what can be said about simplicity of Lyapunov spectrum? Since our results are stated in this general context, it follows that for \textsl{all} unperturbed orbits of a dynamical system, and for almost all choices of the perturbations, we obtain universal gaps between consecutive Lyapunov exponents. By contrast, standard methods based on ergodic theory would likely yield results for almost all orbits. Finally, we remark that our methods are highly robust, and should apply to other kinds of absolutely continuous noise, including situations where the structure of the noise is dependent on the point in the dynamical system.

%\gf{I moved the ``may not exist'' sentence to here because I thought it was a little jarring where it was before.}
Since we place no stationarity assumption
on the original sequence, the limits defining the Lyapunov exponents may not exist. 
Specifically, we show the following.

\begin{mainthm}\label{mt:2x2}
    Let $0<\ep<1$. 
    There exists $c_2(\ep)>0$ such that 
    for any sequence $(A_n)$ of $2\times 2$ matrices, each of norm at most 1, 
    the random variable
    $$
    \liminf_{n\to\infty}\frac 1n\log
    \frac {s_1(A_\ep^{(n)})}
    {s_2(A_\ep^{(n)})}
    $$
    is almost surely constant; and the constant is at least $c_2(\ep)$.
    Further $c_2(\ep)>\exp(-1/\ep^{35})$ for sufficiently small $\ep$.
\end{mainthm}

\begin{mainthm}\label{mt:3x3}
    Let $0<\ep<1$. 
    There exists $c_3(\ep)>0$ such that 
    for any sequence $(A_n)$ of $3\times 3$ matrices, each of norm at most 1, 
    and for $j=1$ or $2$, the random variable
    $$
    \liminf_{n\to\infty}\frac 1n\log
    \frac {s_j(A_\ep^{(n)})}
    {s_{j+1}(A_\ep^{(n)})}
    $$
    is almost surely constant; and the constant is at least $c_3(\ep)$.
    Further $c_3(\ep)>\exp(-1/\ep^{867})$ for sufficiently small $\ep$.
\end{mainthm}

In the case $d>3$, we have the following theorem showing that the Lyapunov spectrum is (quantitatively) non-trivial.
\begin{mainthm}\label{mt:dxd}
Let $d>3$ and $0<\ep<1$. 
    There exists $c_d'(\ep)>0$ such that 
    for any sequence $(A_n)$ of $d\times d$ matrices, each of norm at most 1, 
    the random variable
    $$
    \liminf_{n\to\infty}\frac 1n\log
    \frac {s_1(A_\ep^{(n)})}
    {s_d(A_\ep^{(n)})}
    $$
    is almost surely constant; and the constant is at least $c_d'(\ep)$.
    For small $\ep$, $c_d'(\ep)>\exp(-1/\ep^{16d+3})$.
\end{mainthm}

\begin{mainthm}\label{mt:ort+unif}
   Let $d>1$ be arbitrary. If the unperturbed $d\times d$ matrices $A_n$ are orthogonal for each $n\in\NN$, then there exists a constant $c(\ep)>0$ depending only on $\ep$ such that for each $1\leq j<d$ we have
   \begin{align*}
    \liminf_{n\to\infty}\frac1n\log\frac{s_j(A_\ep^{(n)})}{s_{j+1}(A_\ep^{(n)})}\geq c(\ep).
   \end{align*}
   For any $\delta>0$, for all small $\ep$, $c(\ep)>\exp(-1/\ep^{2+\delta})$.
\end{mainthm}

We remark that this result should be compared with \cite{BednarskiQuas}, where the matrices $(A_n)$
are orthogonal and the perturbations
$\Xi$ are taken to be matrices with independent standard normal entries. In that paper, using symmetry properties of 
multi-variate normal random variables, an exact expression for the Lyapunov exponents is obtained, as well as the leading
terms of a Taylor-like expansion. In any dimension in that setting, the gaps obtained are asymptotically $\ep^2$. 

The next theorem provides universal positive lower bounds on all Lyapunov exponent gaps for perturbed cocycles.
This result
gives no quantitative information
on the gaps, beyond positivity,
 because it relies 
on compactness, and does not exhibit an explicit mechanism for the gaps.
%Finally we state the following theorem, 
\begin{mainthm}\label{mt:dxd2}
Let $d\ge 2$ and $0<\ep<1$. 
    There exists $c_d(\ep)>0$ such that 
    for any sequence $(A_n)$ of $d\times d$ matrices, each of norm at most 1, and for each $1\le j\le d-1$,
    the random variable
    $$
    \liminf_{n\to\infty}\frac 1n\log
    \frac {s_j(A_\ep^{(n)})}
    {s_{j+1}(A_\ep^{(n)})}
    $$
    is almost surely constant; and the constant is at least $c_d(\ep)$.
\end{mainthm}

% This theorem states that there is a universal positive lower bound on the gap size for any dimension, but it
% gives no information beyond the fact that $c_d(\ep)>0$. Unlike the previous results, this result relies 
% on compactness, and does not exhibit an explicit mechanism for the gaps.

We point out that although Theorems A-E are stated for sequences of matrices of norm at most 1, the theorems immediately
apply to sequences of matrices of norm bounded by any constant $K$ if one replaces the gap $c(\epsilon)$ by
the gap $c(\epsilon/K)$.

Also, we do not expect our results to be optimal (or close to optimal): our proofs work by identifying 
very narrow ``targets" that are perturbations of blocks of the original sequence such that each time the randomly perturbed sequence
of matrices passes through one of the targets, one obtains \emph{in expectation} an increment in the logarithmic ratio
of the $j$th and $(j+1)$st singular values. A law of large numbers is then used to give a lower bound for the long-time
almost sure gap. Since we only identify a single target for a block, and do not account for any increases in the ratio of the singular
values that take place other than those occurring as a result of passing through the target perturbation, the actual gap should be
much larger.
It is natural to conjecture that the minimal gaps are obtained in the case where all the unperturbed 
matrices are the identity (since in that
case, the only potential source of growth of the gap between Lyapunov exponents is the random perturbation),
so, given the result of \cite{BednarskiQuas}, one may ask whether it is the case that $c_d(\epsilon)\gtrsim \epsilon^2$.

Finally it is worth pointing out that while we focus on the case of uniform i.i.d.\  perturbations, our methods would also apply 
in the case of absolutely continuous i.i.d.\ perturbations with density bounded away from infinity; 
and away from zero around 0.

\subsection{Context}

The foundational work of Lyapunov \cite{Lyapunov} relates quantities of the form
$$\lim_{t\to\infty} \frac{1}{t}\log\| \varphi(t)v \|$$ 
to stability of solutions of differential equations in $\mathbb R^d$, where $v\in \mathbb R^d$ and
$\varphi(t)$ is a $d\times d$ matrix —the fundamental solution matrix.
These quantities, now called \textit{Lyapunov exponents},
have become fundamental for the study of stability and chaos in dynamical systems. Roughly speaking, negative Lyapunov exponents correspond to stable behaviour and positive Lyapunov exponents are a sign of chaos. 

For systems observed in discrete time, Lyapunov exponents take the form
$$\lim_{n\to\infty} \frac{1}{n}\log\| A_n \dots A_1 v \|,$$ 
where $A_1, A_2, \dots$ are $d\times d$ matrices and $v\in \mathbb R^d$.
Furstenberg and Kesten showed in \cite{FurstenbergKesten} 
that if the matrices $A_j$ are drawn from an ergodic process, then the limit 
$\lim_{n\to\infty} \frac{1}{n}\log\| A_n \dots A_1 \|$ exists and is almost surely constant, although  not necessarily finite. In a similar setting,
Oseledets showed in \cite{Oseledets} that there 
are between 1 and $d$ distinct Lyapunov exponents, $\lambda_1>\dots>\lambda_k$,
with multiplicities $m_1,\ldots,m_k$ summing to $d$,
quantifying the exponential rates of expansion or contraction that different vectors can experience, asymptotically. %where possibly $\lambda_k=-\infty$. 
We refer the reader to the monograph of Viana \cite{VianaBook} for classical theory and recent developments on this topic.

%Because of their asymptotic nature, 
Lyapunov exponents are generally difficult to compute and approximate. 
In fact,
determining whether Lyapunov exponents are non-zero or if there are multiplicities remain difficult problems in general.
% Furstenberg  
% developed a powerful theory to investigate related questions in the case of iid matrix products.
% \aq{Mention Furstenberg: $2\times 2$ case} \cgt{do you mean \cite{FurstenbergKesten,Furstenberg63} as inserted?}
% \aq{Yep}
A $d$-dimensional system is said to have \textit{simple Lyapunov spectrum} if it has $d$ distinct Lyapunov exponents. The question of simplicity of Lyapunov spectrum has received considerable attention in the literature, first in the case of i.i.d.\ matrices and then under increasingly more general stationary driving processes. Works in this direction include
Guivarc’h and Raugi \cite{GuivarchRaugi};  Gol’dsheid and Margulis \cite{GoldsheidMargulis};
Arnold and Cong \cite{ArnoldCong97, ArnoldCong99};
%Bonatti--G\'omez-Mont--Viana \cite{BonattiGMViana},
Bonatti and Viana \cite{BonattiViana04}; Avila and Viana \cite{AvilaViana07-1}; Matheus, M\"oller and Yoccoz \cite{MatheusMollerYoccoz-sl}; Poletti and Viana \cite{PolettiViana-sl}; Backes, Poletti, Varandas and Lima \cite{BPVL-sl}. Establishing simplicity of Lyapunov spectra allowed Avila and Viana to resolve the Zorich–Kontsevich conjecture in \cite{AvilaViana07}. 

The works above build on the theory developed by Furstenberg \cite{Furstenberg63}.
% \flag{To the authors' knowledge}
% \aq{I've checked Avila-Viana. I think that subsumes all of the Poletti and everything Viana is connected to. 
% Gol'dsheid-Margulis and Guivarc'h-Raugi are definitely Furstenberg++. 
% ArnoldCong97 with $L^p$ is cheating! ArnoldCong99 is based on Rokhlin's lemma so is definitely not explicit.
% MMY don't compute gaps. Mitsutani uses AV } , 
This theory has been used to obtain qualitative rather than quantitative results. That is, simplicity is established without providing explicit information on the gaps between Lyapunov exponents. It is also worth mentioning that a non-stationary version of Furstenberg's
theorem for random matrix products has been recently proved by Gorodetski and Kleptsyn in \cite{GorodetskiKleptsyn}.

Despite significant progress, Lyapunov exponents for products of matrices  are still a source of major challenges, even in dimension two. For instance, the study of Lyapunov exponents for two-dimensional maps, such as the H\'enon map and the Chirikov standard map, falls in this category.
Benedicks and Carleson have shown
positive Lyapunov exponents for H\'enon maps in  \cite{BenedicksCarleson}. 
While the problem remains unsolved for the standard map, Blumenthal, Xue and Young in \cite{BlumenthalXueYoung17, BlumenthalXueYoung18},  have established positive lower bounds on Lyapunov exponents for very small stochastic perturbations of the standard map, and other two-dimensional maps. 

%Clearly, in this volume preserving setting, existence of a positive Lyapunov exponent is equivalent to simplicity of the Lyapunov spectrum.

% In this work, we focus on the question of simplicity of Lyapunov spectra for stochastic perturbations of arbitrary products of bounded matrices, 
% from a \textit{quantitative} point of view. 
% %That is, we investigate lower bounds on $\lambda_{j-1} - \lambda_j$. 
% We show that for size $\ep$, uniform and iid  perturbations of any bounded sequence of $2\times 2$ or $3\times 3$ matrices $A_1, A_2, \dots$,  the Lyapunov spectrum is simple, and there is a universal lower bound on the separation of the Lyapunov exponents.
% That is, for $0<\ep<1$ there exists $c(\ep)>0$, independent of the sequence $A_1, A_2, \dots$, such that
% $\lambda^\ep_{j-1} - \lambda^\ep_j > c(\ep)$.

% %\cgt{fill out w/ precise statements.  Maybe even with explicit bounds.}

% The main novelty of this work is that it provides
% \textit{universal} estimates, which do not require any 
% % hyperbolicity properties, or even stationarity,  for the  process generating the sequence of matrices.
% % \cgt{elaborate on significance}
% %Indeed, our results do not require any
% particular (e.g.\ stochastic, deterministic, ergodic, hyperbolic ...) generation process for the original cocycle. The spectral gap estimates hold \emph{uniformly} over \emph{any} sequence of  matrices with a common norm bound. Our proof strategy is described at the start of \S\ref{S: gen strategy}.

From a broader perspective, our results may be contrasted with previous works on 
genericity of trivial Lyapunov spectrum.
The papers of Bochi \cite{Bochi} and of Bochi and Viana \cite{BochiViana} show that for a generic
non-hyperbolic cocycle (over a continuous invertible base), the Lyapunov spectrum is trivial. That is, 
they show that by making very careful perturbations to the cocycle, one can get the Lyapunov 
exponents to collapse. In contrast, our results show that for any initial cocycle  
(including non-stationary) and almost every sequence of random perturbations, the 
Lyapunov spectrum is simple, and in dimension two or three, they provide explicit bounds on 
the gaps, only depending on the perturbation size. 

The key source of difficulty in the general setting with Lyapunov exponents is that singular values may be far 
from multiplicative. One has bounds such as $s_1(A)s_d(B)\le s_1(AB)\le s_1(A)s_1(B)$ which are far too weak
to control limits. Indeed, this issue underlies the theorems of Bochi and Bochi--Viana. On the other hand, in our randomly perturbed setting, singular 
values are approximately multiplicative (as embodied here by Lemma \ref{gluing lemma}) and this is what makes
our theorems work.
This approach builds on previous work of Froyland, Gonz\'alez-Tokman and 
Quas \cite{FGTQ}, and is ultimately inspired by work of Ledrappier and Young \cite{LedrappierYoung91}.
The idea of controlling sub-multiplicative quantities with nearly multiplicative ones 
has also been exploited in other studies of Lyapunov exponents, in the form of an avalanche principle, 
introduced by Goldstein and Schlag in \cite{GoldsteinSchlag} and expanded upon in Duarte and Klein's 
monograph \cite{DuarteKlein}.
Apart from specific cases where Lyapunov exponents have been fully computed (see \cite{BednarskiQuas} and references therein), this work seems to be the first one where explicit  lower bounds on Lyapunov exponent gaps have been found for random perturbations of cocycles with arbitrary Lyapunov spectrum.

The use of \textit{stochastic perturbations} as a tool for understanding dynamical systems has a long history, described e.g. by Kifer in \cite{KiferBook}. 
More recently, the potential of this approach to provide theoretical and practical insights into the long-term behaviour of complex dynamical systems has been highlighted by Young in \cite{Young08,Young13}. 
On one hand, random perturbations are a natural way to model phenomena evolving in the presence of noise.
On the other hand, by taking zero-noise limits, stochastic perturbations have been effectively used to gain information about dynamical systems, since the work of Khas'minskii \cite{Khasminskii}.
Results about stochastic stability and Lyapunov exponents for randomly perturbed dynamical systems
include the works of 
Young \cite{Young86},
Ledrappier and Young \cite{LedrappierYoung91};
Imkeller and Lederer \cite{ImkellerLederer};
Baxendale and Goukasian \cite{BaxendaleGoukasian};
Cowieson and Young \cite{CowiesonYoung};
Lian and Stenlund \cite{LianStenlund};
Froyland, Gonz\'alez-Tokman and Quas \cite{FGTQ, FGTQ19};
Blumenthal and Yun \cite{BlumenthalYun};
Chemnitz and Engel \cite{ChemnitzEngel};
Bednarski and Quas \cite{BednarskiQuas}.
%\cgt{others?}.
Progress on the related question of continuity of Lyapunov exponents has been recently reviewed by Viana in \cite{VianaContReview}.

% In recent work, Bednarski and Quas \cite{BednarskiQuas}, find the explicit values of $\lambda^\ep_{j-1} - \lambda^\ep_j$ for orthogonal cocycles in dimension $d$ perturbed by i.i.d.\ normal perturbations of size $\ep$, and conjecture that these gaps provide a universal lower bound, of order $\ep^2$, for bounded
% normally perturbed  cocycles.  
% A related result in the present work is proving a lower bound for $\lambda^\ep_{j-1} - \lambda^\ep_j$ -- see Theorem~\ref{mt:ort+unif} -- in dimension $d$ when the original cocycle is orthogonal and the perturbations are i.i.d.\ uniform of size $\ep$.

\subsection*{Acknowledgments}
This research has been funded by the Australian Research Council (ARC DP220102216) and NSERC. 
The authors thank Alex Blumenthal and Paulo Varandas for bibliographic suggestions.
We thank the referee for a careful reading and many helpful suggestions.

\section{Preliminaries and Notation}

For a $d\times d$ matrix $M$, we write $\|M\|$ for its operator norm, that is $\|M\|=\sup_{\|x\|_2=1}\|Mx\|_2$ % \cgt{should we say which one?}
and $|M|_\infty$ for $\max|m_{ij}|$.
Let $(A_n)_{n\in\NN}$ be a fixed sequence of $d\times d$ matrices of (operator) norm at most 1 and
let $(\Xi_n)_{n\in\NN}$ be a family of independent identically distributed $d\times d$ 
matrix random variables with mutually independent entries,
each distributed as $U([-1,1])$. Let $A_{n,\ep} = A_n+\ep \Xi_n$ for $\ep>0$. 
Define $A^{(n)}:=A_nA_{n-1}\cdots A_1$ and (for $\ep>0$) 
$A_\ep^{(n)}:=A_{n,\ep}A_{n-1,\ep}\cdots A_{1,\ep}$.

For a $d\times d$ matrix $A$, let
$s_1(A)\geq \dots\geq s_d(A)$ be its singular values and for $1\leq j\leq k\leq d$ set $S_j^k(A)=s_j(A)\cdots s_k(A)$ be the product of the $j$th singular value through the $k$th singular value.  
%Set $N=\blocksize$ and set 
% \begin{align}\label{eq reqgapindex}
%     \rgapindex > \frac{4d(1+\log(C_*d\ep^{-1}))}{-p\log\ep}
% \end{align} for some values $C_*>0$ and $p\in(0,1)$ to be specified later.
 %We will consider blocks of length $\blocklength-1$: For $k\geq 0$, let $B_k=A_{(k+1)N-1,0} \dots A_{kN+1,0}$.
For non-zero vectors $u$ and $v$ in $\RR^d$, we define $\angle(u,v)=\big\|u/\|u\|-v/\|v\|\big\|$. 
For subsets $C_1$ and $C_2$ of $\RR^d$, we define $d(C_1,C_2)=\inf_{x\in C_1,y\in C_2}\|x-y\|$. If the sets are disjoint,
one is compact and the other is closed, this quantity is positive and the infimum is attained. 
We denote the unit sphere in $\mathbb{R}^d$ by $S$.

\begin{lemma}\label{lem:projbound}
    Let $\RR^d$ be the expressed as a direct sum (not necessarily orthogonal) 
    $E\oplus F$ and let $\Pi_{E\parallel F}$ and $\Pi_{F\parallel E}$ be projections
    of $\RR^d$ onto $E$ and $F$ respectively so that for any $x\in\RR^d$, 
    $\Pi_{E\parallel F}(x)+\Pi_{F\parallel E}(x)$ is the unique decomposition of $x$
    in $E\oplus F$. Let $\delta=d(E\cap S,F\cap S)$ where
    $S$ is the unit sphere. 
    Then $\|\Pi_{E\parallel F}\|$ and $\|\Pi_{F\parallel E}\|$ are at most $2/\delta$.    
    \end{lemma}
%    \cgt{Anthony, is this also our definition of angle? Namely, is $\angle(E,F)=\min\{\|e-f\|\colon e\in E\cap S,\ f\in F\cap S\}$? Even if we are going to get rid of angles later, do we agree this is what we mean in the current version?}
%    \aq{This is what we called $\perp(E,F)$ in at least one previous paper.}
\begin{proof}
    Let $x\in S$ and
write $x$ as $e-f$. We then have $\big\|e/\|e\|-f/\|e\|\big\|= 1/\|e\|$
and $\big\|f/\|e\|-f/\|f\|\big\|=\big|\|f\|/\|e\|-1\big|=\big|\|f\|-\|e\|\big|/\|e\|
\le \|f-e\|/\|e\|=1/\|e\|$. Hence by the triangle inequality we see 
$\big\|e/\|e\|-f/\|f\|\big\|\le 2/\|e\|$. Since $e=\Pi_{E\parallel F}(x)$,
this gives $\|\Pi_{E\parallel F}(x)\|\le 2/\big\|(e/\|e\|)-(f/\|f\|)\big\|\le 2/\delta$.
\end{proof}
We observe that this proof does not make use of the fact that $\RR^d$ is Euclidean,
so it applies in arbitrary normed spaces.

% \gf{Check if $d(\cdot,\cdot)$ defined later, if not, define here.}
% \begin{lemma}\label{lem:linspacebound}
%     Let $u$ and $v$ be unit vectors in $\RR^d$. Then
%     $d(u,\lin(v))\ge 
%     \frac 12\min(\|u+ v\|,\|u-v\|)$. 
% \end{lemma}

% \begin{proof}
%     Let $t$ be the real value for which $\|u-tv\|$ is minimized. We 
%     assume (replacing $v$ by $-v$ if necessary) that $t\ge 0$. 
%     We then have that $\|tv-u\|\ge |\|tv\|-\|u\||=|t-1|=\|tv-v\|$,
%     where the first inequality used the reverse triangle inequality. 
%     It follows that $\|u-v\|\le \|u-tv\|+\|tv-v\|\le 2\|u-tv\|$ as required.
% \end{proof}

\begin{lemma}\label{lem:linspacebound}
Let $u\in S$ be a unit vector in $\RR^d$ and let $V$ be a subspace of $\RR^d$. 
Then $d(u,V\cap S)\le 2d(u,V)$. 
\end{lemma}

\begin{proof}
    If $d(u,V)=1$, the conclusion follows from the triangle inequality. 
    Otherwise, let $v\in V\setminus\{0\}$ be such that $d(u,V)=\|u-v\|$. Then
    $d(u,V\cap S)\le \big\|u-v/\|v\|\big|\le \|u-v\|+\big\|v-v/\|v\|\big\|
    =\|u-v\|+\big|\|v\|-1\big|=\|u-v\|+\big|\|v\|-\|u\|\big|\le 2\|u-v\|
    =2d(u,V)$.    
\end{proof}

This proof also works in normed spaces. In a Euclidean space, the constant
can be improved from $2$ to $\sqrt 2$. 
%\ja{I think this observation/remark was from a previous iteration of this lemma and can be deleted?}
%\aq{Good catch. I have taken reciprocals of the constants and I think it's now correct.
% I'm keen on leaving the comment in there in some form because there are various definitions of
% distance between spaces out there and it's helpful to have the conversion lemmas around.}
\section{General Strategy}\label{S: gen strategy}

In this section we present the general strategy for proving Theorems \ref{mt:2x2} to \ref{mt:ort+unif}. 
%, Theorem \ref{thm main result}.
We begin with the following definition: 
Given a $d\times d$ matrix $A$ and $1\leq j< k\leq d$, we say that $A$ has a 
\textsl{$\gap$-large $(j,k)$-gap} (or \textsl{$(j,k)$-gap} for short)
if $\frac{s_j(A)}{s_k(A)}\geq \reqgap$, where $\reqgap$ is a constant to be determined depending on $\ep$. 

The main idea of the proof is to break the sequence of matrices $(A_n)$ into blocks of length 
$\blocklength$, $B_i:=A_{(i+1)N}\cdots A_{iN+1}$, $i\geq 0$, where $\blocklength$ is a constant 
to be determined based on $\ep$.
We then show for each such block $B_i$, how to construct a \emph{target block} $\target{B}_i$ 
where $\target{B}_i$ is of the form $\target{A}_{(i+1)N}\cdots \target{A}_{iN+1}$ where 
$|\target{A}_l-A_l|_\infty$ is small for each $l$ and the target blocks $\target{B}_i$
have a $(j,k)$-gap. 

First, in \S\ref{S:nonsing} we present a procedure for perturbing a matrix to make it 
sufficiently non-singular. The full details of the target block construction will be 
given in \S\ref{S: find target}. In 
\S\ref{S: triangle ineq}, for a given target block $\target{B}_i$ with a $(j,k)$-gap, we show that any 
block $\tarpert{B}_i$ sufficiently close to $\target{B}_i$ must also have a $(j,k)$-gap. 
%In \S\ref{S: bookkeeping} we define \textit{good blocks} to be those perturbed blocks $B_{i,\ep}:=A_{(i+1)N,\ep}\cdots A_{iN+1,\ep}$
%that are sufficiently close to the target block $\target{B}_i$, and \textit{bad blocks} to be everything between the good blocks. 
%In particular, we show that the set of perturbations producing good blocks has positive measure 
%with respect to the normalized product measure on the space of perturbations. 
Then in \S\ref{S: gluing},  
%building on ideas from \cite{FGTQ}, and ultimately inspired by \cite{LedrappierYoung}
%(see also \cite{DuarteKlein})
we show that, when two blocks of any length where the random perturbations have been fixed
are separated by a single matrix where the random perturbation is not yet determined, 
the logarithmic $(j,k)$ gap of the combined block is the sum of logarithmic gaps of the outside 
two blocks plus a random variable with a tight distribution. The gap size $\gap$ is chosen to ensure that
the gains from hitting targets dominate the losses from this joining procedure. 
In \S\ref{S: Kolmogorov}, 
we use the Kolmogorov 0--1 law to show that the growth rate of the ratio between the $j$th and $k$th singular values of partial products is constant almost surely. 
In \S\ref{S: bkkpng}, we present a book-keeping procedure, Theorem~\ref{thm:meta}, relying on the gluing lemma (Lemma \ref{gluing lemma}), 
which provides a quantitative lower bound on the exponential growth rate of 
$s_j(A^{(n)}_\epsilon)/s_k(A^{(n)}_\epsilon)$ based on the frequency of hitting 
neighbourhoods of the target blocks. 
%to concatenate together all of the good and bad blocks. 
%Since we have that good blocks should occur with some positive frequency, we are able to find a positive 
%lower bound on the $(j,k)$-gap of the sequence of perturbed matrices. 
%Finally, in \S\ref{S: complement}, we show that if we can build a $(j,j+1)$-gap then we can build a $(d-j, d-j+1)$-gap. 

\subsection{Non-Singular Initialization}\label{S:nonsing}
% \aq{I'd like to move this to \S4. The rationale is that I think the general strategy should be there,
% while the more specific stuff about finding targets should be in \S5.
% Or maybe even move it inside the book-keeping proof.}\cgt{sounds good to me.}
% \aq{So now I'm not so sure. I think it depends whether we see this as something to be done in every
% dimension as part of the target-building; or once as part of the general strategy. There are reasons for
% both. At the moment, I think I prefer the idea of doing it once in each dimension.}
%\gf{Mention, perhaps around here that the  matrix norm $\|\cdot\|$ is induced by the the vector $2$-norm, while the matrix norm $\|M\|_\infty=\max_{1\le i,j\le d} |m_{ij}|$. We could consider denoting our matrix infinity norm as $|M|_\infty$ to indicate it is not an induced norm.}\aq{I implemented your $|\cdot|_\infty$ suggestion. The defn is now in the preliminary section}.
Our target blocks in Lemma \ref{lem: tri ineq} will be products of non-singular matrices.
The following lemma provides a simple way to construct non-singular matrices near possibly singular matrices. 
\begin{lemma}\label{lem:nonsinginit}
Let $d\geq2$. Let $A$ be a $d\times d$ matrix with $\|A\|\leq 1$ and $0<\ep<1$. 
Then, there exists  $A'$ such that $\|A'-A\|<\frac\ep2$, $\|A'\|\le 
\max(\frac 12,1-\frac\ep2)$ 
and $s_d(A')\geq \frac\ep2$. 
\end{lemma}
\begin{proof}
%By the pigeonhole principle, there exists at least one interval among the $d+1$ intervals $\{[0,\ep/2d), [\ep/2d,2\ep/2d), \dots [d\ep/2d,(d+1)\ep/2d)]\}$ which does not contain a singular value of $A$.
%Let $c$ be the center of such an interval.

%We assume without loss of generality that $\ep\le 1$. 
%\cgt{can't we assume $\ep\leq 1$ in the statement?}\aq{Done.}
Let $A=VDU^T$ be a singular value decomposition of $A$.
Then, $A' = V \big((1-\ep)D+\ep(\frac12)\big) U^T$ satisfies the required property.
Indeed, if $s_1, \dots, s_d$ are the singular values of $A$, then the singular values of $A'$ are 
$(1-\ep)s_1+\ep\frac 12, \dots ,(1-\ep)s_d+\ep\frac12$.
Also, $\|A'-A\|\le\frac\ep2$.
\end{proof}
In fact by \cite[Lemma 1]{Raghunathan}, the singular value decomposition may be done in a measurable way, so that $U$, $D$ and $V$
may be chosen to depend measurably on $A$. Hence, $A'$ may be taken to be a measurable function of $A$.

\subsection{Triangle Inequality for Target Block}\label{S: triangle ineq}
%\ja{Our good blocks may be sub-blocks with length less than $N$, so I think we need to change the wording in this lemma. }

The following continuity lemma shows that under arbitrary suitably small
perturbations, the singular value ratios for subproducts of matrices of length less 
than $N$ cannot decrease by more than half. 
\begin{lemma}\label{lem: tri ineq}
Let $d>1$, $1\le m\le m'\le N-1$ and $0<\ep<1$.
Suppose $\target{A_m}, \dots, \target{A_{m'}}$, called targets,  are $d\times d$ 
matrices of norm at most $1-\frac\ep4$ such that $s_d(\target{A_i})\geq \frac\ep4$ 
for every $m\leq i\leq m'$.
%with $N=N(\ep)$ such that 
%$\lim_{\ep\to 0}N(\ep)=\infty$.
%$N=\lceil \ep^{-7\rgapindex}\rceil$ \cgt{we just need $N\to \infty$ as $\ep\to 0$, but I've copied condition from pairwise far...}.  
For each $m\leq i\leq m'$,
let $\tarpert{A}_i$ be such that $|\tarpert{A}_i - \target{A_i}|_\infty< \tinygap$. 
%\aq{Do we want to say $|\tarpert{A_i}-\target{A_i}|_\infty<\tinygap$ instead of $\|\tarpert{A_i}-\target{A_i}\|<\tinygap$?
%The advantage would be we can then write down an expression for $p$. I don't think the proof would change other than adding an extra
% factor of $\sqrt d$ when we estimate $\|\tarpert{A_i}-\target{A_i}\|$.}\cgt{sounds good, I've tried to update.}
Then, for every $1\leq j<k\leq d$ we have
\begin{equation}\label{eq:svPert}
\frac{s_j(\tarpert{A_{m'}} \dots \tarpert{A}_m)}{s_{k}(\tarpert{A_{m'}} \dots \tarpert{A}_m)}\geq \frac12\,
\frac{s_j(\target{A_{m'}} \dots \target{A}_m)}{s_{k}(\target{A_{m'}} \dots \target{A}_m)}.
\end{equation}
\end{lemma}

\begin{proof}
%Fix $1\leq j<k\leq d, 1\leq m < n \leq \blocklength$.
% Note that for each $1\leq \ell \leq d$, %and $1\leq k\leq N$,
% $|s_\ell(\tarpert{A_n} \dots \tarpert{A}_m) - s_\ell(\target{A_{n}}\cdots\target{A_{m}})| \leq \|\target{A_{n}}\cdots\target{A_{m}}- \tarpert{A_{n}}\cdots\tarpert{A_{m}}\|$.
Note that $\|\tarpert{A_i}-\target{A_i}\|\le d|\tarpert{A_i}-\target{A_i}|_\infty<(\frac\ep 4)^N/(3N)$.
Since $\|\tarpert {A_i}\|\le \|\target{A_i}\|+\|\tarpert{A_i}-\target{A_i}\|$,
we have $\|\tarpert{A_i}\|$ and $\|\target{A_i}\|$ are at most 1 for each $m\le i\le m'$. 
%Recall that $\|\target{A_i}\|, \|\tarpert{A_i}\|\leq 1$ for all $m\leq i\leq m'$.
Thus, the triangle inequality implies
\begin{align*}
\| \target{A_{m'}}\cdots\target{A_{m}}-\tarpert{A_{m'}}\cdots\tarpert{A_{m}}\| &\leq
\sum_{i=m}^{m'} \| \target{A_{m'}}\cdots \target{A_{i}} \tarpert{A}_{i-1} \dots \tarpert{A_m} -
\target{A_{m'}}\cdots \target{A_{i+1}} \tarpert{A}_i \dots \tarpert{A}_m\| \\
&=
\sum_{i=m}^{m'} \|\target{A_{m'}}\cdots\target{A_{i+1}}(\target{A_i}-\tarpert{A}_i)\tarpert{A_{i-1}}\dots\tarpert{A_m}\|\\
&\leq 
\sum_{i=m}^{m'} \|  \target{A_{i}}  -
\tarpert{A}_i \| \leq (\tfrac{\ep}4)^{N}/3.
\end{align*}
%where we used the inequality $\|M\|\le d|M|_\infty$ to obtain  $\|A_i^\times\|\le 1$, $i=m,\ldots,m'$.
%\aq{Put in $d$ rather than $\sqrt d$ because $\|A\|\le d|A|_\infty$.}
Using the well-known fact\footnote{This follows from the so-called max-min inequality:
$s_j(A)=\max_{\dim V=j}\min_{v\in V\cap S} \|(A-A'+A') v\| \leq \|A-A'\|+ \max_{\dim V=j}
\min_{v\in V\cap S} \|A'v\|=\|A-A'\|+s_j(A')$.} that $|s_j(A)-s_j(A')|\le \|A-A'\|$ for all $j$, we see
\begin{align*}
\frac{s_j(\tarpert{A_{m'}} \dots \tarpert{A}_m)}{s_{k}(\tarpert{A_{m'}} \dots \tarpert{A}_m)}
&\ge \frac{s_j(\target{A_{m'}} \dots \target{A}_m)- (\tfrac{\ep}4)^{N}/3}{s_{k}(\target{A_{m'}} \dots \target{A}_m)+
(\tfrac{\ep}4)^{N}/3}\\
&=\frac{s_j(\target{A_{m'}} \dots \target{A}_m)}{s_{k}(\target{A_{m'}} \dots \target{A}_m)}
\left(\frac{1-(\tfrac{\ep}4)^{N}/(3s_j(\target{A_{m'}} \dots \target{A}_m))}
{1+(\tfrac{\ep}4)^{N}/(3s_k(\target{A_{m'}} \dots \target{A}_m))}\right).
\end{align*}
% Using the mean value theorem for the function $f(x)=\frac{a-x}{b+x}$, where $a\geq b>0$, we get that for $|x|<b$,
% $f(x)\geq \frac{a}{b} \left(1-\frac{2b|x|}{(b-|x|)^2} \right)$.
% Letting $a= s_{j}(\target{A_{n}}\cdots\target{A_{m}}), b=s_{k}(\target{A_{n}}\cdots\target{A_{m}})$ and $x=\| \target{A_{n}}\cdots\target{A_{m}}-\tarpert{A_{n}}\cdots\tarpert{A_{m}}\|$ 
%  we get, for sufficiently small $\ep>0$,
% $$\frac{s_j(\tarpert{A_\blocklength} \dots \tarpert{A}_1)}{s_{k}(\tarpert{A_\blocklength} \dots \tarpert{A}_1)}\geq \frac{s_j(\target{A_n} \dots \target{A}_m)}{s_{k}(\target{A_n} \dots \target{A}_m)}\left( 1-\frac{2 s_{k}(\target{A_{n}}\cdots\target{A_{m}}) \blocklength\tinygap}{(s_{k}(\target{A_{n}}\cdots\target{A_{m}})-\blocklength\tinygap)^2}\right).$$
%\gf{How do we know e.g.\ that $s_j(\tarpert{A_\blocklength} \dots \tarpert{A}_1)\ge s_j(\target{A_n} \dots \target{A}_m) - x$ to make the replacement in the function $f$ above?}\cgt{is the footnote ok?}
%\aq{I added the inequality that we need in the text.}
Since $s_j(\target{A_{m'}} \dots \target{A}_m)$ and $s_k(\target{A_{m'}} \dots \target{A}_m)$
are at least $s_d(\target{A_{m'}})\cdots s_d(\target{A_m})\ge (\frac\ep 4)^N$, 
we obtain
$$
\frac{s_j(\tarpert{A_{m'}} \dots \tarpert{A_m})}{s_{k}(\tarpert{A_{m'}} \dots \tarpert{A}_m)}
\ge 
\frac{s_j(\target{A_{m'}} \dots \target{A}_m)}{s_{k}(\target{A_{m'}} \dots \target{A}_m)}
\left(\frac{1- \tfrac 13}
{1+\tfrac 13}\right),
$$
so that the desired inequality holds. 
\end{proof}

We denote by $\PP$ the i.i.d.\ probability measure on sequences of matrices with independent
uniformly distributed entries in $[-1,1]$.
The probability that the perturbed matrices are close enough along a sampled block to 
satisfy the hypothesis of Lemma \ref{lem: tri ineq} has a straightforward lower bound.
\begin{corollary}\label{cor pos measure}
Let $d>1$. For all $0<\ep<1$ there exists $p\in (0,1)$  (for example $p=(\tfrac \ep4)^{N^2d^2}/(3dN)^{Nd^2}$ works)
such that for any target,
$\target{A_m}, \dots, \target{A_{m'}}$ as in Lemma~\ref{lem: tri ineq} with $1\le m\le m'\le N-1$,
% \gf{Maybe I missed it, but I think below is the first instance of $\mathbb{P}$...it would be good to explicitly define it before here.}
% \aq{done}
% Let $1\leq m\le m'< N$, $\target{A}=(\target{A_m},\cdots,\target{A_{m'}})$, and define
% \begin{align*}
%     \Dl(\target{A})=\{(\tarpert{A_m},\cdots, \tarpert{A_n}):\|\tarpert{A_i}-\target{A_i}\|<\tinygap \text{ for all }m\leq i\leq n\}.
% \end{align*}
$$
\mathbb P\big(|A_{i,\ep}-\target{A_i}|_\infty<\tinygap \text{ for all $i=m,\ldots,m'$})\big)\ge p.
$$
%(\Dl(\target{A}))\geq p$,
%for all blocks $A = A_n\cdots A_m$ and all  $1\leq m<n\leq N$. 
%If $A_m^\ep,\ldots,A_{m'}^\ep$ satisfy this condition then the conclusion of
%Lemma \ref{lem: tri ineq} holds.
%\gf{Do we need this final sentence?}
%
%$(\tarpert{A_m},\cdots, \tarpert{A_n})\in \Dl(\target{A})$, 
%\eqref{eq:svPert} holds.
\end{corollary}

\subsection{Gluing Lemma}\label{S: gluing}

% \aq{I take back my suggestion about dividing through by $|\log\ep|$ in the definition of $Z(A)$. It's just
% too problematic when $\ep$ is close to 1. In any case what's really important is that for fixed
% $\ep$, the $Z(A)$ are tight. On the other hand, I still suggest (1) emphasizing the dependence of $Z$ on $A,L,R$ and
% maybe $\ep$, $j,k$ also; and (2) defining $Z_{j,k}(L,A,R)$ to be \emph{equal} to $\log q_{j,k}(L(A+\ep\Xi)R)-
% \log q_{j,k}(L)-\log q_{j,k}(R)$. 
% Doing this leads to a rather beautiful expression for $\log q_{j,k}({A^\ep}^{(n)})$:
% write ${A^\ep}^{(n)}$ as $B_lT_{l-1}B_{l-1}\cdots B_2T_1B_1$ where the $B$'s are either
% targets that have been hit, or blocks between targets that are hit and the $T$'s are single transition matrices. 
% Then $\log q_{j,k}({A^\ep}^{(n)})=\log q_{j,k}(B_1)+\ldots+\log q_{j,k}(B_l)+Z_1+\ldots Z_{l-1}$
% where the $Z_i$'s are the random variables $Z(B_{i+1},A_{t_i},P_i)$ and 
% $P_i=B_iT_{i-1}\cdots B_2T_1B_1$.}

For $1\leq j<k\leq d$ and a $d\times d$ matrix $M$, let $q_{j,k}(M)=\frac{s_j(M)}{s_k(M)}$. 
We define a function $F$ whose arguments are $d\times d$ matrices by
$F(L,A,R)=\log q_{j,k}(LAR)-\log q_{j,k}(L)-\log q_{j,k}(R)$. 
This function is used to measure
how close to multiplicative $q_{j,k}$ is. 
%\aq{Mention Ledrappier--Young somewhere here as a precursor to [8]}\cgt{added a sentence at the start of the section.}

\begin{lemma}\label{gluing lemma}
Let $d>1$ and let $1\le j<k\le d$. 
% Given fixed matrices $L,R,A$ with $L$ and $R$ non-singular and $\|A\|\leq 1$ and a random matrix $\Xi$ whose entries $\xi_{ij}$ are 
% independent and uniformly identically distributed in $[-1,1]$, 
% then for each $1\leq j<k\leq d$ the random variables $F(L,A+\ep\Xi,R)$
%$$
%    F(L,A,R,\Xi,j,k):=\log q_{j,k}(L(A+\ep\Xi)R)-\log q_{j,k}(L)-\log q_{j,k}(R)
%$$
%there exists a random variable $Z(A)$ and a constant $K>0$ such that  
%\begin{align*}
%    \log q_{j,k}(L(A+\ep\Xi)R)\geq \log q_{j,k}(L)+\log q_{j,k}(R)+Z(A).
%\end{align*} 
Then there exists $\zt>0$ such that for all invertible $d\times d$ matrices $L$ and $R$ and
for all $A$ such that $\|A\|\le 1$, all $X\geq 0$ and all $0<\ep<1$, we have   
\begin{equation*}
    \PP(|F(L,A+\ep\Xi,R)|>X)\leq \min\{1,\zt e^{-\frac{X}{4d}}\ep^{-1}\},
\end{equation*}
where $\Xi$ is a matrix-valued random variable with independent $U[-1,1]$ entries.
\end{lemma}

\begin{proof}
    Recall that $S_1^\ell(A)=s_1(A)\cdots s_\ell(A)$ for each $1\leq \ell\leq d$, and note that 
    \begin{equation}\label{q_j equality}
        \log q_{j,k}(A)=\log S_1^j(A)-\log S_1^{j-1}(A)+\log S_1^{k-1}(A)-\log S_1^k(A).     
    \end{equation}
    Using the fact that $\|A\|\leq 1$, and the fact that the entries of $\Xi$ are all in $[-1,1]$, we first note that 
    \begin{equation}\label{chi_j up bd}
        \log S_1^\ell(L(A+\ep\Xi)R)
        \leq 
        \log S_1^\ell(L)+\log S_1^\ell(R)
        +d\log(1+\ep d).
    \end{equation}
    
    Following the proof of Lemma 3.5 in \cite{FGTQ} we write $L=O_1D_1O_2$ where $O_1$ and $O_2$ are orthogonal and $D_1$ is diagonal with non-negative entries arranged in decreasing order. Similarly we write $R=O_3D_2O_4$, where $O_3$ and $O_4$ are orthogonal and $D_2$ is diagonal. Let $A'=O_2AO_3$ and $\Xi'=(1/d)O_2\Xi O_3$ and $C_\ell$ be the diagonal matrix with $1$'s in the first $\ell$ diagonal positions and $0$'s elsewhere. 
    (The choice of normalization of $\Xi'$ is to ensure that its entries remain in $[-1,1]$).
    In this case we have 
    \begin{align*}
        &S_1^\ell(L(A+\ep\Xi)R)=S_1^\ell(D_1(A'+\ep d\,\Xi')D_2),\\
        \quad
        &S_1^\ell(L)=S_1^\ell(D_1),
        \quad\text{ and }\quad
        S_1^\ell(R)=S_1^\ell(D_2).
    \end{align*}
We briefly summarize the next steps of the proof of Lemma 3.5 in \cite{FGTQ}.
Using the facts that $S_1^\ell(OM)=S_1^\ell(M)=S_1^\ell(MO)$ for an orthogonal matrix $O$
and $S_1^\ell(C_\ell M),S_1^\ell(MC_\ell)\le S_1^\ell(M)$, we have
$S_1^\ell(L(A+\ep\Xi) R)=S_1^\ell(O_1D_1O_2(A+\ep\Xi)O_3D_2O_4)
=S_1^\ell(D_1O_2(A+\ep\Xi)O_3D_2)\ge 
S_1^\ell(C_\ell D_1O_2(A+\ep\Xi)O_3D_2C_\ell)$. 
Now since $C_\ell D_1=C_\ell D_1C_\ell$ and since $S_1^\ell$ is multiplicative
for matrices whose non-zero entries are in the top left $\ell\times\ell$ submatrix, we 
obtain
    \begin{equation} 
    \label{chi_j low bd}
        \log S_1^\ell(L(A+\ep\Xi)R)
        \geq 
        \log S_1^\ell(L)+\log S_1^\ell(R)
        +\log S_1^\ell(C_\ell(A'+\ep d\,\Xi')C_\ell).
    \end{equation}
Combining \eqref{chi_j up bd} and \eqref{chi_j low bd}, we deduce for each $\ell\le d$ we have
\begin{equation}\label{chi_j abs}
\Big|\log S_1^\ell(L(A+\ep\Xi)R)-\log S_1^\ell(L)-\log S_1^\ell(R)\Big|
\le d\log(1+\ep d)+|\log S_1^\ell(C_\ell(A'+\ep d\,\Xi')C_\ell)|.
\end{equation}
    
For each $1\leq \ell\leq d$ let $A_\ell''$ and $\Xi_\ell''$ be the top-left $\ell\times \ell$ submatrices of $A'$ and $\Xi'$ respectively and note that $S_1^\ell(C_\ell(A'+\ep d\,\Xi')C_\ell)=|\det (A_\ell''+\ep d\,\Xi_\ell'')|$.

Hence \eqref{q_j equality} gives
\begin{equation}
    |F(L,A+\ep\Xi,R)|\le 4d\log(1+\ep d)+\sum_{\ell\in\{j-1,j,k-1,k\}}\Big|\log |\det (A_\ell''+\ep d\,\Xi_\ell'')|\Big|.
\end{equation}

In order to have $|F(L,A+\ep\Xi,R)|>X$, it is therefore necessary that one of the
$\big|\log|\det(A_\ell''+\ep d\,\Xi_\ell'')|\big|$ terms is greater that $\frac X4-d\log(1+\ep d)$.
%\aq{I've added some words to hopefully clear up the appearance of repetition here.}\cgt{this looks good to me}

Since $A_\ell''+\ep d\,\Xi_\ell''=C_\ell O_2(A+\ep\Xi)O_3C_\ell$,
$\|A_l''+\ep d\,\Xi_\ell''\|\le \|A+\ep\Xi\|\le 1+\ep d$ for each 
$\ell\le d$, so that $\log|\det(A_\ell''+\ep d\,\Xi_\ell'')|
\le d\log(1+\ep d)$ for all $\Xi$ and all $\ell\le d$. 
We have shown in the case $X>8d\log(1+\ep d)$, that it is impossible for $\log|\det(A_\ell''+\ep d\,\Xi_\ell'')|$
to exceed $\frac X4-d\log(1+\ep d)$.
Hence in order to have
$|F(L,A+\ep\Xi,R)|>X$, it is necessary that one of the 
$\log|\det(A_\ell''+\ep d\,\Xi_\ell'')|$ terms is \emph{less than} $-\big(\frac X4-d\log(1+\ep d)\big)$.

That is for $X>8d\log(1+\ep d)$, 
\begin{equation}\label{eq:splitprob}
    \begin{split}
    &\PP(|F(L,A+\ep\Xi,R)|>X) \\
    &\le 4\max_{\ell\in\{j-1,j,k-1,k\}}
    \PP\Big(\log|\det(A_\ell''+\ep d\,\Xi_\ell'')|<-\big(\tfrac X4-d\log(1+\ep d)\big)\Big).
    \end{split}
\end{equation}

The distribution of $\Xi_\ell''$ is absolutely continuous with respect to the uniform measure on $\ell\times \ell$ matrices
with entries in $[-1,1]$, 
and has a bounded density, where the bound only depends on $d$.
%\aq{I don't see how to prove this.
%One issue is that the entries are in $[-\sqrt d,\sqrt d]$, I think. Also we should probably say a bit
%more about why the density is bounded ( I know we didn't in the CPAM paper!)}
%\cgt{we've scaled the entries of $\Xi$ to $1/\sqrt{d}$; we should double check to see if this affects the previous proof or anything else in the paper. Or perhaps easier to just keep the $\sqrt{d}$ in this paragraph's bound?}
Indeed, the linear map $T: \mathbb R^{d^2} \to \mathbb R^{d^2}$ given by $T(\Xi)=O_2 \Xi O_3$ is an isometry in the Frobenius norm, so the density of 
$\Xi'$ (thought of as a vector in $[-1,1]^{d^2}\subset \mathbb R^{d^2}$) is uniform on its image.
Thus, by taking marginals,  the density of the top-left $\ell\times \ell$ submatrix of $\Xi'$ is scaled by at most a factor of  $2^{d^2-\ell^2}$ in each coordinate.

% \cgt{About the bound on the density, we think the distribution of $\Xi'$ is uniform on its image, but we didn't show this. Even if this is not true, I think boundedness follows from the linear map $T: \mathbb R^{d^2} \to \mathbb R^{d^2}$ given by $T(A)=O_1 A O_2$ being invertible for all choices of orthogonal $O_1, O_2$; indeed, $T^{-1}(B)=O_1^{-1} B O_2^{-1}$. A uniform bound follows from compactness.}

Now, from the proof of Lemma 3.5 of \cite{FGTQ}, there
exists $C>0$ such that for all $\ell\le d$, for every $\ell\times \ell$ matrix $B$ of 
$\|\cdot\|$-norm at most 1, and for all $K>0$,
\begin{equation*}
    \PP(\{Z_\ell\colon \log|\det(B+\ep Z_\ell)|<-K\})\leq \min\{1, Ce^{-K/d}\ep^{-1}\},
\end{equation*}
where the random variable $Z_\ell$ is an $\ell\times\ell$ matrix random variable with 
independent entries uniformly distributed in $[-1,1]$.

Substituting the terms of \eqref{eq:splitprob} into this, replacing $\ep$ with $\ep d$ and absorbing the factors 4, $d$ and the upper bound for
the density into $C$, we obtain
$$
\PP(|F(L,A+\ep\Xi,R)|>X)\le Ce^{-(X/4-d\log(1+\ep d))/d}\ep^{-1}.
$$
which is of the required form.
\end{proof}

\begin{corollary}\label{cor:SLLN applies}
For $\lambda>0$, let $\Exp(\lambda)$ denote an exponentially distributed random variable with parameter $\lambda$.
    In the setting of Lemma \ref{gluing lemma}, the random variables
    $|F(L,A+\ep\Xi,R)|$ are uniformly stochastically dominated by the 
    random variable $4d\log(\frac\zeta\ep)+\Exp((4d)^{-1})$ which has exponentially
    decaying tails.

    In particular, there is a constant $K$ (depending only on $d$)
    such that for any $0<\ep<1$, any invertible $L$ and $R$, and any $A$ of norm at most 1, 
    $\EE (F(L,A+\ep\Xi,R))\ge 4d\log\ep-K$.
\end{corollary}

\begin{proof}
    Observe that if $Y\sim 4d\log(\frac\zeta\ep)+\Exp((4d)^{-1})$, then for any $X>0$, 
    $\PP(Y>X)=\min(1,\exp[-(4d)^{-1}(X-4d\log(\frac\zeta\ep))])=\min(1,\zeta e^{-\frac{X}{4d}}\ep^{-1})$.

    Hence for any $L$, $R$ and $A$ as in the statement of Lemma \ref{gluing lemma}, 
    and $\Xi$ a matrix random variable with independent $U[-1,1]$ entries,
    \begin{align*}
        \EE F(L,A+\ep\Xi,R)\ge -|\EE F(L,A+\ep\Xi,R)|\ge -4d\log\tfrac\zeta\ep-4d=4d\log\ep-K,
    \end{align*}
    where $K=-4d(\log\zeta+1)$.
\end{proof}
% where the constant $C$ does no
% for some $C>0$, where we have replaced $\log S_1^\ell(A_\ell''+\ep \Xi_\ell'')$ with $-\log\det(A_\ell''+\ep \Xi_\ell'')$ since $\log S_1^\ell$ agrees with the logarithm of the absolute value of the determinant.

% \aq{This is $\PP(-F>K)$. I think we need upper bounds also $\PP(F>K)<e^{-K/d}$ sort of thing}

% Thus, using \eqref{q_j equality}, \eqref{chi_j up bd}, and \eqref{chi_j low bd} we have 
% \begin{align*}
%     &\log q_{j,k}(L(A+\ep\Xi)R)
%     \\
%     &\quad
%     \geq\log q_{j,k}(L)+\log q_{j,k}(R)+\log S_1^j(C_j(A'+\ep\Xi')C_j) +\log S_1^{k-1}(C_{k-1}(A'+\ep\Xi')C_{k-1})-2d\log(1+\ep d)
%     \\
%     &\quad
%     =\log q_{j,k}(L)+\log q_{j,k}(R)-\log|\det(A_j''+\ep\Xi_j'')|-\log|\det(A_{k-1}''+\ep\Xi_{k-1}'')|-2d\log(1+\ep d).
% \end{align*}
% Thus we have  
% $$
%     F(L,A+\ep\Xi,R,)\geq -\lt(\log|\det(A_j''+\ep\Xi_j'')|+\log|\det(A_{k-1}''+\ep\Xi_{k-1}'')|+2d\log(1+\ep d)\rt).
% $$
% Noting that for $K$ large, if $F(L,A+\ep\Xi,R)>K$ then $-\log\det|A_\ell''+\ep\Xi_\ell''|>K/2$ for either $\ell=j$ or $\ell=k$, we must have that 
% \begin{align*}
%     \PP(F(L,A+\ep\Xi,R)>K)\leq \min\{1, Cde^{-\frac{K}{2d}}\ep^{-1}\},
% \end{align*}
% which completes the proof. 

In the sequel we will make use of the following theorem.
\begin{theorem}[Theorem 2.19 of \cite{HH}]\label{thm SLLN}
    Let $(Z_n)$ be a sequence of random variables and $\sF_n$ an increasing sequence of 
    $\sg$-algebras such that $Z_n$ is measurable with respect to $\sF_n$ for each $n\geq 1$. 
    Suppose that there exists a random variable $Z$ with $\EE(|Z|\log^+|Z|)<\infty$ and a 
    constant $c>0$ such that for each $X>0$ and each $n\geq 1$  
    \begin{align}\label{SLLN cond}
        \PP(|Z_n|>X)\leq c\PP(|Z|>X).
    \end{align} 
    Then the Strong Law of Large Numbers holds for $Z_n$, i.e. 
    \begin{align*}
        \frac{1}{n}\sum_{k=1}^n(Z_k-\EE(Z_k\rvert\sF_{k-1}))
        \xrightarrow{a.s.} 0
    \end{align*}
    as $n\to\infty$.
\end{theorem}

% \aq{I think it would make sense to have a note here pointing out that the r.v.s in L4.3 satisfy the 
% hypothesis of the above? Maybe also obtain a uniform lower bound for the expectation?}

\subsection{Constancy of limits}\label{S: Kolmogorov}
We now show that the quantities appearing in Theorems \ref{mt:2x2}-\ref{mt:dxd2} are almost surely constant.
\begin{lemma}\label{lem:gap a.s. const}
    Let $d\ge 2$, let $(A_n)$ be an arbitrary sequence of $d\times d$ matrices of 
    norm at most 1 and let $\ep>0$. Then for any $1\le j<k\le d$, the quantity
    $\liminf_{n\to\infty} \frac 1n\log \big(s_j(A_\ep^{(n)})/s_{k}(A_\ep^{(n)})\big)$ is almost surely constant.     
\end{lemma}

% Given this, we call the almost sure value of the
% quantity in the statement $g_j\big((A_n),\ep\big)$. 

\begin{proof}
    The proof is an application of the Kolmogorov 0--1 law. 
    Write %$S_a^b(A)$ for the product $s_a(A)\cdots s_b(A)$ and 
    $A_{i,i',\ep}$ for the product $A_{i',\ep}\cdots A_{i+1,\ep}$.
    We rely on the following well-known facts about singular values:
    firstly $s_1(A)s_d(B)\le s_1(AB)\le s_1(A)s_1(B)$ for any matrices $A$ and $B$;
    and secondly $s_1(A^{\wedge l})=S_1^l(A)$ 
    and $s_D(A^{\wedge l})=S_{d-l+1}^d(A)$ where $A^{\wedge l}$ is the $l$-fold
    exterior power of $A$ acting on the $l$-fold exterior power of $\RR^d$
    and $D=\binom dl$ is the dimension of the $l$-fold exterior power of $\RR^d$.

    Let $m$ be an arbitrary fixed natural number. We have $A_{0,m,\ep}$ is almost surely non-singular.
    Let $c$ and $C$ be random variables depending only on $\Xi_1,\ldots,\Xi_m$ defined by
    $c=\min_{i,j}S_i^j(A_{0,m,\ep})$
    and $C=\max_{i,j}S_i^j(A_{0,m,\ep})$ so that $0<c<C<\infty$.
    For $n>m$, and any $1\le l\le d$, applying the above inequalities to $(A_{0,n,\ep})^{\wedge l}=
    (A_{m,n,\ep})^{\wedge l}(A_{0,m,\ep})^{\wedge l}$,
    %\gf{changed to $A_{m,n,\ep}^{\wedge k}(A_\ep^{(m)})^{\wedge k}$} 
    we obtain
 $$
        S_1^l(A_{m,n,\ep})S_{d-l+1}^d(A_{0,m,\ep})\le S_1^l(A_{0,n,\ep}) \le
        S_1^l(A_{m,n,\ep})S_{1}^l(A_{0,m,\ep}),     
$$
so that for each $l$, we have
$$
cS_1^l(A_{m,n,\ep})\le S_1^l(A^{(n)}_\ep)\le CS_1^l(A_{m,n,\ep})
$$
for all $n$.
Dividing the $l$th system of inequalities by the $(l+1)$st, we obtain for each $l$,
$$
\tfrac{c}Cs_l(A_{m,n,\ep})\le s_l(A^{(n)}_\ep)\le \tfrac{C}cs_l(A_{m,n,\ep}),
$$
and hence, dividing the $j$th system of these inequalities by the $k$th, we obtain
$$
\frac{c^2}{C^2}\frac{s_j(A_{m,n,\ep})}{s_{k}(A_{m,n,\ep})}\le \frac{s_j(A_\ep^{(n)})}{s_{k}(A_\ep^{(n)})}
\le \frac{C^2}{c^2}\frac{s_j(A_{m,n,\ep})}{s_{k}(A_{m,n,\ep})}.
$$
    % As observed above, \gf{minor point, $q_j$ as defined above was a logarithmic quantity, so we should perhaps call the earlier definition $\log q_j$.}
    % $$
    % q_j(A)=\frac{S_1^j(A)^2}{S_1^{j-1}(A) S_1^{j+1}(A)},
    % $$ so that
    % $$
    % \frac{S_{d-j+1}^d(A_\ep^{(m)})^2}
    % {S_1^{j-1}(A_\ep^{(m)})S_1^{j+1}(A_\ep^{(m)})} q_j(A_{m,n,\ep})
    % \le
    % q_j(A_\ep^{(n)})\le \frac{S_1^j(A_\ep^{(m)})^2}
    % {S_{d-j+2}^d(A_\ep^{(m)})S_{d-j}^d(A_\ep^{(m)})}q_j(A_{m,n,\ep})
    % $$
    % Since $S_{d-j+1}^d(A_\ep^{(m)})^2/
    % (S_1^{j-1}(A_\ep^{(m)})S_1^{j+1}(A_\ep^{(m)}))$ and
    % $S_1^j(A_\ep^{(m)})^2
    % /(S_{d-j+2}^d(A_\ep^{(m)})S_{d-j}^d(A_\ep^{(m)}))$ are finite and non-zero,
    Taking logs, and dividing by $n$, we see
    $$
    \liminf_{n\to\infty} \frac 1n \log \frac{s_j(A_\ep^{(n)})}{s_{k}(A_\ep^{(n)})}=
\liminf_{n\to\infty} \frac 1n \log \frac{s_j(A_{m,n,\ep})}{s_{k}(A_{m,n,\ep})}.
    $$
    Letting $\mathcal F_m$ be the smallest $\sigma$-algebra
    with respect to which the sequence of perturbations $\Xi_{m+1},\Xi_{m+2},
    \ldots$ is measurable, we deduce that the quantity
    $\liminf_{n\to\infty} \frac 1n \log \big(s_j(A_\ep^{(n)})/s_{k}(A_\ep^{(n)})\big)$ 
    is $\mathcal F_m$-measurable. Since $m$ was arbitrary, we deduce that
    $\liminf_{n\to\infty} \frac 1n \log \big(s_j(A_\ep^{(n)})/s_{k}(A_\ep^{(n)})\big)$ is tail-measurable.
    Since the perturbations are independent, it follows from the Kolmogorov 0--1 law
    that $\liminf_{n\to\infty} \frac 1n\log \big(s_j(A_\ep^{(n)})/s_{k}(A_\ep^{(n)})\big)$
    is almost surely constant as required.
\end{proof}

\subsection{Bookkeeping} \label{S: bkkpng}

The main result in this section is a meta-theorem, showing that if every block
of matrices of a fixed length has a nearby ``target'' block with a large gap between 
its $j$th and $k$th singular values, then there is a uniform lower bound on the
gap between the $j$th and $k$th Lyapunov exponents of the cocycle $A_\ep^{(n)}$
depending only on $\ep$.

Somewhat more precisely, the hypothesis is that for every 
$\ep>0$, there exists an $N$ such that for every block $M_1,\ldots,M_{N-1}$ 
of $d\times d$ matrices, there exists a \emph{target} $\target{M_m},\ldots\target{M_{m'}}$
with $1\le m\le m'\le N-1$
consisting of nearby matrices,
for which there is a \emph{gap} between the $j$th 
and $k$th singular values (that is $s_j(\target{M_{m'}}\cdots\target{M_m})/
s_k(\target{M_{m'}}\cdots\target{M_m})$ exceeding a threshold determined based on the
Gluing Lemma above).
The conclusion is that there is an explicit lower bound $c(\ep)>0$,
such that for all sequences $(A_n)_{n\in\NN}$ of $d\times d$ matrices of norm 1,
$\liminf_{n\to\infty}\frac 1n\log \big(s_j(A_\ep^{(n)})/s_k(A_\ep^{(n)})\big)
\ge c(\ep)$.

If $0<\ep<1$ is fixed, we say that $\tilde M$ is a \emph{near-identity perturbation} of
$M$ if $\tilde M=RM$ or $MR$ for some matrix $R$ satisfying $\|R-I\|<\frac\ep4$ and $\|R^{-1}-I\|<\frac\ep4$.

% \gf{Change $\gap>1$ to $\gap>2$ in statement of Theorem \ref{thm:meta} to match Lemma \ref{lem:tar3} and Corollary \ref{cor:tar3}?}
% \aq{Done}
\begin{theorem}\label{thm:meta}
    Let $d>1$ be fixed and let $1\le j<k\le d$. 
    Suppose that for all $0<\ep<1$ and all $\gap>2$, there exists an $N\in\NN$
    such that for every sequence $M_1,\ldots,M_{N-1}$ of invertible $d\times d$ matrices
    %$1-\frac\ep2$
    %and with $s_d(M_i)\ge\frac\ep 2$ for each $i$,
    there exists a \emph{target}: 
    numbers $1\le m\leq m'\le N-1$ and near-identity perturbations $\target{M_m},
    \ldots,\target{M_{m'}}$ of $M_m,\ldots,M_{m'}$ such that 
    $$
    \frac{s_j(\target {M_{m'}}\cdots \target{M_m})}{s_k(\target{M_{m'}}\cdots\target{M_m})}>\gap.
    $$
%        where $\target{M_i}=R_iM_i$ or  
         % \cgt{I changed the gap to $\gap$ (with a macro). I've also updated the target section to work for general gaps (not powers of epsilon).}
         % \aq{The new structure (including the $\gap$) is looking very nice to me...}\cgt{and now also has blocks of N-1 matrices, not N!!}
         % \aq{\faThumbsOUp}
%    \end{itemize}
    Then there exists $c(\ep)>0$ such that for any sequence $(A_n)_{n\in\NN}$ of $d\times d$
    matrices of norm at most 1, the almost sure constant $\liminf_{n\to\infty}\frac 1n
    \log\big(s_j(A_\ep^{(n)})/s_k(A_\ep^{(n)})\big)$ satisfies
    \begin{equation*}
        \liminf_{n\to\infty}\frac 1n\log\frac{s_j(A_\ep^{(n)})}{s_k(A_\ep^{(n)})}>c(\ep).
    \end{equation*}
    Furthermore, the asymptotic behaviour of the lower bound $c(\ep)$ can be explicitly 
    computed based on the expectations of the random variables appearing in Section \ref{S: gluing}.
\end{theorem}

\begin{proof}
    Let $d,j$ and $k$ be as in the statement of the theorem and let $0<\ep<1$. 
        We recall the definition 
    \begin{equation}\label{eq:logsvdiff}
    F(L,A,R)=G(LAR)-G(L)-G(R),
    \end{equation}
    where $G(M)=\log q_{j,k}(M)$. 

    Let $\lambda=\inf_{L,A,R}\EE_\Xi F(L,A+\ep\Xi,R)$ where the random variable 
    $\Xi$ runs over matrices with independent $U[-1,1]$-valued entries,
    and the infimum 
%    
%    Let $\lambda=\inf_{L,A,R}\EE Z^\ep_{j,k}(L,A,R)$ where the random variables
%    $Z^\ep_{j,k}(L,A,R)$ are as in Section \ref{S: gluing} and the infimum
    is taken over all non-singular matrices $L$ and $R$ and all matrices $A$
    of norm at most 1. By Corollary \ref{cor:SLLN applies}, $\lambda\ge 4d\log\ep-K$,
    where the constant $K$ depends only on $d$.
%    The fact that $\lambda$ is finite is a consequence
%    of Lemma \ref{gluing lemma}. 
    We then apply
    the hypothesis of the theorem with $\gap$ taken to be $2e^{2+2|\lambda|}$
    to obtain an $N$ such that for every $(N-1)$-block,
    $M_{1},\ldots,M_{N-1}$
    there is a target $\target{M_{m}},\ldots,\target{M_{m'}}$. 
%\aq{We should at least obtain asymptotic bounds on $\lambda$ in terms of $d$ and $\ep$, 
%and maybe bounds that work for all $d$ and all $\ep$ in some range.}
    
    Starting with a sequence of matrices $A_1,A_2,\ldots$, each of norm at most 1,
    we initially make perturbations of size at most $\frac\ep 2$ as in Lemma
    \ref{lem:nonsinginit} to obtain matrices $A'_1, A'_2,\ldots$ each of norm at most 
    $1-\frac\ep 2$ and with smallest singular value at least $\frac\ep 2$.
    % \cgt{NOTATION. In \S~\ref{S:nonsing}, the roles of $A_1^0$ and $A_1$ are opposite to here. Maybe change here?}
    % \aq{Oops.. Done}
    % \aq{I'm having a hard time deciding whether it would be better to do the initial perturbations
    % in S5 when we build the targets, or whether it's better here. The advantage of doing it in S5 is
    % that T4.5 can just say $|A_n|<=1$ and not say anything about $s_d$ so that the T4.5 statement 
    % is more streamlined. The disadvantages are
    % (1) we need to say $\|\target{M_i}-M_i\|\le 3\ep/4$ rather than $\ep/4$ in T4.5;
    % and (2) we have to apply L5.1 separately in both d=2 and d=3.}

    For each $(N-1)$-block $A'_{lN+1},\ldots,A'_{lN+N-1}$ where $l$ ranges over $\NN_0$, we next apply 
    the hypothesis of the theorem to produce a target block $\target{A_{lN+m_l}},\ldots,\target{A_{lN+m_l'}}$
    with $1\le m_l<m_l'\le N-1$ 
    where $\target A_i$ is a near-identity perturbation of $A_i'$
    for each $i\in [lN+m_l,lN+m'_l]$ such that the
    product has at least the required gap of $2e^{2+2|\lambda|}$ between the $j$th and $k$th singular values.
    Note that if $\target{A_i}=R_iA_i'$, then $\|\target{A_i}-A_i'\|=\|(R_i-I)A_i'\|<\frac\ep 4$ 
    (and similarly if $\target{A_i}=A_i'R_i$), which ensures that $\|\target{A_i}-A_i\|<\frac{3\ep}4$,
    $\|\target A_i\|<1-\frac\ep 4$     and $s_d(\target{A_i})>\frac\ep 4$. 

    As established in Lemma \ref{lem: tri ineq} if the 
    random sub-block $A_{lN+m_l,\ep},\ldots,A_{lN+m_l',\ep}$ falls in the
    \emph{target area}, that is, if $|A_{i,\ep}-\target{A_i}|_\infty\le \tinygap$ for
    each $lN+m_l\le i\le lN+m_l'$, then the ratio of the
    $j$th and $k$th singular values of $A_{lN+m_l',\ep}\cdots A_{lN+m_l,\ep}$
    exceeds $e^{2+2|\lambda|}$. 
    Since $|\target{A_i}-A_i|_\infty\le \|\target{A_i}-A_i\|<\frac{3\ep} 4$,
    the $(\tfrac\ep4)^N/(3dN)$-ball in the $|\cdot|_\infty$-norm around $\target{A_i}$
    lies within the $\ep$-ball around $A_i$. That is, the target area is 
    contained in the support of $A_{i,\ep}$.
    % \gf{?}
    % \aq{I think I fixed the error this time. 
    % (Please delete my comment if the fix is OK).}
    Hence each target is independently hit
    with probability at least $p=(\tinygap)^{d^2N}=(\tfrac\ep4)^{d^2N^2}/(3dN)^{d^2N}$. 
    The quantity $c(\ep)$ will end up being $p/N$.
    
    Note that the $m_l$ and $m_l'$ are determined solely by the unperturbed sequence
    of matrices (that is, they are deterministic). We let $\mathcal B_0$ denote the $\sigma$-algebra generated by
    $\{\Xi_i\colon i\in\bigcup_l [lN+m_l,lN+m_l']\}$. That is $\mathcal B_0$
    is the smallest $\sigma$-algebra with respect to which the matrix random variables
    $\Xi_i$ in the target sub-blocks are measurable. 
    We then define a $\mathcal B_0$-measurable random subset of $\NN$ by 
    \begin{equation*}
        \mathsf{TargHit}=\bigcup_{l\in \mathbb N_0}\{[lN+m_l,lN+m'_l]\colon |A_{i,\ep}-\target{A_i}|_\infty<\tinygap\text{ for all }i\in [lN+m_l,lN+m'_l]\}.
    \end{equation*}
    That is, $\mathsf{TargHit}$ consists of the indices of those matrices in blocks where the target is successfully hit. 
    Note that since $1\le m_l<m'_l<N$ for each $l$, there is a gap of length at least 1 between successive targets
    that are hit. We now define a second subset called $\mathsf{Trans}$ (for transition indices) 
    consisting of those $i\in\NN$ immediately preceding or following a block in $\mathsf{TargHit}$. This is
    again $\mathcal B_0$-measurable. We enumerate the elements of $\mathsf{Trans}$ as
    $T_1<T_2<T_3<\ldots$, so that the indices $(T_n)$ are $\mathcal B_0$-measurable
    random variables. Lemma \ref{lem:gap a.s. const} implies that the limit in the statement is
    almost surely constant. 
    
    % By construction $T_{i+1}\ge T_i+2$ for each $i$ (as there is a target that
    % is hit in between). \aq{A minor annoyance. It just might be possible to have two targets 
    % exactly 2 apart, so that there are two consecutive transition coordinates. This is actually a non-issue: you just need an identity matrix in between}.
    
    Further, we assume without loss of generality (using Lemma \ref{lem:gap a.s. const} again) that $T_1>1$
    and define $T_0=0$. 

    At this point we introduce a second $\sigma$-algebra: $\mathcal F_0$ which is generated by 
    $\{\Xi_i\colon i\not\in\mathsf{Trans}\}$. This is a refinement of $\mathcal B_0$ since all of the
    transition coordinates lie outside the target sub-blocks (each transition coordinate
    immediately precedes or follows a target sub-block). We claim that conditioned on 
    $\mathcal F_0$, the random variables $\Xi_{T_i}$ are independent and identically distributed
    with entries uniformly distributed over $[-1,1]$. The reason for this is that the $T_i$ are
    $\mathcal B_0$-measurable (hence $\mathcal F_0$-measurable) and the $\Xi_{T_i}$ are independent of
    all of the $\Xi_m$'s that generate $\mathcal F_0$. (Informally, we could say that we never ``looked at''
    the $\Xi_{T_i}$'s when identifying the transition coordinates, so that the $\Xi$ values at these times are
    independent of everything that we have looked at). 

    We are now ready to define a sequence of $\sigma$-algebras that will be used when we apply Theorem 
    \ref{thm SLLN}: Let $\mathcal F_n=\mathcal F_0\vee \sigma(\Xi_{T_1},\ldots,\Xi_{T_n})$. 
    That is, $\mathcal F_n$ is the smallest $\sigma$-algebra such that all non-transition $\Xi$ variables
    as well as the first $n$ transition $\Xi$ variables are measurable. By the observation above, $\Xi_{T_n}$
    is independent of $\mathcal F_{n-1}$. 

    We name the blocks appearing in $A_\ep^{(n)}$, defining 
    $B_{i,\ep}=A_{T_i-1,\ep}\cdots A_{T_{i-1}+1,\ep}$. (In the unusual case where $T_i=T_{i-1}+1$, $B_{i,\ep}$ is just the identity). 
    With this notation, 
    $$
    A_\ep^{(n)}=R_{n,\ep} A_{T_{L_n},\ep}B_{L_n,\ep}A_{T_{L_n-1},\ep}\cdots A_{T_2,\ep}B_{2,\ep}
    A_{T_1,\ep}B_{1,\ep},
    $$
    where $L_n=\#(\mathsf{Trans}\cap \{1,\ldots,n\})$ is an $\mathcal F_0$-measurable random variable 
    and $R_{n,\ep}$ is the tail $A_{n,\ep}\cdots A_{T_{L_n}+1,\ep}$. 
    Using \eqref{eq:logsvdiff}, we see $G(BAC)=G(B)+G(C)+F(B,A,C)$. Applying
    it twice, first separating out $B$ and $A_1$, we obtain
\begin{align*}    
    G(BA_1CA_2D)&=G(B)+G(CA_2D)+F(B,A_1,CA_2D)\\
    &=G(B)+G(C)+G(D)+F(C,A_2,D)+F(B,A_1,CA_2D).
\end{align*}
    Using induction, we have
    \begin{equation}\label{eq:Gsum}
    \begin{split}
    G(A_\ep^{(n)})&=G(B_{1,\ep})+\ldots+G(B_{L_n,\ep})+G(R_{n,\ep})+
F(R_{n,\ep},A_{T_{L_n},\ep},\bar B_{L_n,\ep})\\
    &+
    \sum_{i=1}^{L_n-1}F(B_{i+1,\ep},A_{T_i,\ep},\bar B_{i,\ep})
    ,   
    \end{split}
    \end{equation}
    where $\bar B_{i,\ep}=B_{i,\ep}A_{T_{i-1},\ep}\cdots A_{T_1,\ep}B_{1,\ep}$.
    Since each transition coordinate is immediately followed or preceded by a target that
    was hit,
    at least half of the $B_{i,\ep}$ consist of blocks where the target area
    was hit. By Lemma \ref{lem: tri ineq}, $G(B_{i,\ep})\ge 2+2|\lambda|$ for those 
    blocks. For the other blocks, $G(B_{i,\ep})\ge 0$ by definition.
    Hence we see
    \begin{equation}\label{eq:Gbound}
        \liminf_{n\to\infty}\frac{1}{L_n}\sum_{i=1}^{L_n}G(B_{i,\ep})\ge 1+|\lambda|\text{ a.s.}
    \end{equation}

    We let $Z_i=F(B_{i+1,\ep},A_{T_i,\ep},\bar B_{i,\ep})$ and note that it is $\mathcal F_i$-measurable.
    However since $\Xi_{T_i}$ is independent of $\mathcal F_{i-1}$, 
    $$
    \EE\big(F(B_{i+1,\ep},A_{T_i}+\ep\Xi_{T_i},\bar B_{i,\ep})|\mathcal F_{i-1}\big)
    =\EE_{\Xi}F(B_{i+1,\ep},A_{T_i}+\ep\Xi,\bar B_{i,\ep}).
    $$
    That is, $\EE(Z_i|\mathcal F_{i-1})$ is exactly the expectation of a random variable as in 
    Lemma \ref{gluing lemma}, so that by definition, $\EE(Z_i|\mathcal F_{i-1})\ge \lambda$.

    Hence by the version of the Strong Law of Large Numbers, Theorem \ref{thm SLLN},
    \begin{equation}\label{eq:Fbound}
    \liminf_{n\to\infty}\frac{1}{L_n}\sum_{i=1}^{L_n-1}F(B_{i+1,\ep},A_{T_i,\ep},\bar B_{i,\ep})\ge\lambda\text{ a.s.}
    \end{equation}
    Each time a target is hit, at least one transition
    index is created (two such indices are created unless hits are adjacent).
    Since each successive target is independently hit with probability at least $p$,
    we may apply 
    Theorem \ref{thm SLLN} one more time (with $\sigma$-algebras $\mathcal F_n'=\sigma(\{
    \Xi_i\colon i\in \bigcup_{l=0}^{n-1}\{lN+1,\ldots,lN+N-1\}\})$) to obtain
    \begin{equation}\label{eq:Lbound}
        \liminf_{n\to\infty}\frac{L_n}n\ge \frac pN\text{ a.s.}
    \end{equation}
    Adding \eqref{eq:Gbound} and \eqref{eq:Fbound}; and then multiplying by \eqref{eq:Lbound},
    we obtain 
   $$ \liminf_{n\to\infty}\frac 1n \left(\sum_{i=1}^{L_n}G(B_{i,\ep})+
    \sum_{i=1}^{L_n-1}F(B_{i+1,\ep},A_{T_i,\ep},\bar B_{i,\ep})\right)
    \ge \frac pN\text{ a.s.}
    $$
    Comparing this with \eqref{eq:Gsum}, it remains to control $G(R_{n,\ep})$ and 
    $F(R_{n,\ep},A_{T_{L_n},\ep},\bar B_{L_n,\ep})$.
    By definition, $G(R_{n,\ep})$ is non-negative. 
    Finally we will show that $\frac 1nF(R_{n,\ep},A_{T_{L_n},\ep},\bar B_{L_n,\ep})\to 0$ a.s.
    Using \eqref{eq:Gsum} and the inequalities above, this is sufficient to lead to the conclusion
    $$
    \liminf_{n\to\infty} \frac 1n \log\frac{s_j(A_\ep^{(n)})}{s_k(A_\ep^{(n)})}\ge \frac pN\text{ a.s.}
    $$
    To show that $\frac 1nF(R_{n,\ep},A_{T_{L_n},\ep},\bar B_{L_n,\ep})\to 0$ a.s., let $\eta>0$ be arbitrary. Then
    \begin{align*}
    &\PP(|\tfrac 1nF(R_{n,\ep},A_{T_{L_n},\ep},\bar B_{L_n,\ep})|>\eta)\\
    &=\EE\big(\PP(|\tfrac 1nF(R_{n,\ep},A_{T_{L_n},\ep},\bar B_{L_n,\ep})|>\eta|\mathcal F_{L_n-1})\big)\\
    &=\EE\big(\PP(|F(R_{n,\ep},A_{T_{L_n},\ep},\bar B_{L_n,\ep})|>n\eta|\mathcal F_{L_n-1})\big)\\
    &=\EE\big(\PP_\Xi(|F(R_{n,\ep},A_{T_{L_n},\ep}+\ep\Xi,\bar B_{L_n,\ep})|>n\eta|\mathcal F_{L_n-1})\big).
    \end{align*}
    By the bounds in Lemma \ref{gluing lemma}, the above quantities are summable, so that by the first
    Borel-Cantelli lemma, almost surely $|\tfrac 1nF(R_{n,\ep},A_{T_{L_n},\ep},\bar B_{L_n,\ep})|\le\eta$ for all but finitely
    many $n$. Since $\eta$ is an arbitrary positive number, we see 
    $\tfrac 1nF(R_{n,\ep},A_{T_{L_n},\ep},\bar B_{L_n,\ep})\to 0$ a.s.\ as required.
\end{proof}

We also state an alternative formulation of Theorem \ref{thm:meta} that will be useful in Section
\ref{S:abstract}.

\begin{corollary}\label{cor:metareform}
    Let $d>1$ be fixed and let $1\le j<k\le d$. 
    Let $0<\ep<1$ and suppose that for all $\gap>2$, there exists an $N\in\NN$
    such that for every sequence $M_1,\ldots,M_{N-1}$ of $d\times d$ matrices of norm 
    at most $1-\frac\ep2$ and $d$th singular value at least $\frac\ep 2$,
    %$1-\frac\ep2$
    %and with $s_d(M_i)\ge\frac\ep 2$ for each $i$,
    there exists a \emph{target}: 
    numbers $1\le m\leq m'\le N-1$ and perturbations $\target{M_m},
    \ldots,\target{M_{m'}}$ of $M_m,\ldots,M_{m'}$ such that $|\target{M_i}-M_i|_\infty\le\frac\ep{4d}$
    for each $m\le i\le m'$ and 
    $$
    \frac{s_j(\target {M_{m'}}\cdots \target{M_m})}{s_k(\target{M_{m'}}\cdots\target{M_m})}>\gap.
    $$
    Then $c(\ep)>0$, where $c(\ep)$ is as in the statement of Theorem  \ref{thm:meta}.
\end{corollary}

The proof is almost exactly the same as the proof of Theorem \ref{thm:meta}. 
Unlike in that theorem, there is no need to apply Lemma \ref{lem:nonsinginit}.
The new hypothesis ensures that $\|\target{M_i}-M_i\|\le\frac\ep 4$ for each $m\le i\le m'$
which is exactly the conclusion that was obtained from the ``near identity perturbation"
hypothesis in the proof of Theorem \ref{thm:meta}. From that point in the proof onwards, that
is the only fact about the $\target{M_i}$ that is used, so that the remainder of the proof
of Theorem \ref{thm:meta} applies directly.

\section{Concrete targets}\label{S: find target}
This section identifies suitable targets for our general strategy put forth in \S\ref{S: gen strategy}. 
First, in \S\ref{S: complement},
we show that if we can build a $(j,j+1)$ $\gap$-large gap then we can build a $(d-j, d-j+1)$ $\gap$-large gap.  
In \S\ref{S:far}, we give a general argument to identify targets in the case where singular vectors remain \textit{spread} as they evolve within a block. 
Finally, \S\ref{S:2dtarget} and \S\ref{S:3dtarget} complete the target identification step for the case of $2\times 2$ -- or more generally, extremal singular values-- and $3\times3$ matrices, respectively.

Let $B=\mr{M_{\blocklength-1}}\cdots \mr{M_{1}}$ 
be a product of $\blocklength-1$ invertible matrices.
%where each matrix is of norm at most $1-\frac\ep 2$
%and has $d$th singular value at least $\frac\ep2$. 
Let the singular vectors of $B$ be $v_1,\ldots,v_d$,
listed as usual in order of decreasing singular values. 
If two or more singular values coincide, we just fix once and for all a
choice of orthonormal singular vectors. 
Let $E_j(B)$ denote the span of $(v_i)_{i\le j}$ and $F_j(B)$ denote the span of $(v_i)_{i>j}$.
For $1\le j<d$ and for each $1\le n< \blocklength$, we let
$E_{j,n}$ denote the linear span of %\gf{should the $A$ on this line have circle superscripts?} 
$(\mr{M}_n\ldots \mr{M}_1v_i)_{i\le j}$
and let $F_{j,n}$ denote the span of $(\mr{M}_n\ldots \mr{M}_1v_i)_{i>j}$.
% We let $v_1,v_2,v_3$ be singular vectors of $B$ and 
We write
$v^n_i$ for the unit normalization of $M_n\cdots M_1v_i$. 
Let $\delta_{n,j}$ denote the minimal distance between norm-1
vectors in $E_{j,n}$ and $F_{j,n}$ and let $\delta_n=\min_{1\leq j < d}\delta_{n,j}$.
We will say the block $\mr{M_{\blocklength-1}},\ldots ,\mr{M_{1}}$
is $\eta$-\emph{spread} if $\delta_n>\eta$ for each $1\le n<\blocklength$
and \emph{nearly aligned} otherwise. As mentioned above, in the case where some singular
values have multiplicity, spreadness or nearly alignedness may depend on the choice
of singular vectors, so formally in our definition, these properties depend not only
on the matrices, but also on the choice of singular vectors.
% \gf{small suggestion, how about the terminology ``nearly $\eta$-aligned'' since we sometimes use different levels of alignedness}
% \aq{The term ``nearly $\eta$-aligned" makes me feel a bit queasy. I am happy with the informal terminology as I think it's 
% clear when it's used, but I would rather say ``not $\eta$-spread" rather than ``nearly $\eta$-aligned" if we feel we need precision.}
% \gf{Ok, I will remove the comment.}

\subsection{Complementarity} \label{S: complement}

Recall that for given a matrix $M$, a fixed $\ep>0$, and a matrix $R$ satisfying
$\|R^{\pm 1}-I\|<\frac\ep4$,
we call $RM$ or $MR$ 
a \emph{near-identity perturbation}.

Let $d\in\NN$ and $\ep>0$ be fixed.
We assume there is a
quantity $\gap$, depending on $d$ and $\ep$, called the 
\emph{required gap size}. 
We say that we can \emph{build a $\gap$-large gap between the $j$th and $(j+1)$st singular values}
if there is an $N$ such that 
for every block $M_1,\ldots,M_{N-1}$ of invertible $d\times d$ matrices,
there exists a sub-block $M_{m},\ldots,M_{m'}$ with $1\le m\leq m'\le N-1$
and targets $\target{M_{m}},\ldots,\target{M_{m'}}$ consisting of near-identity
perturbations of the sub-block
such that $s_j(\target{M_{m'}}\cdots\target{M_m})/s_{j+1}(\target{M_{m'}}\cdots\target{M_m})>\gap$. 
In this case, we shall say that the block $\target{M_{m'}}\cdots\target{M_m}$
has a \emph{$\gap$-large $(j,j+1)$-gap}.

\begin{lemma}\label{lem:complementarity}
    Let $d\in\NN$ and $\ep>0$. If for every block of $N-1$ invertible $d\times d$ matrices, we can 
    build a $\gap$-large gap between the $j$th and
    $(j+1)$st singular values, then for every block of $N-1$ invertible $d\times d$ matrices,
    we can build a $\gap$-large gap between the $(d-j)$th and $(d-j+1)$st
    singular values.
\end{lemma}

\begin{proof}
    Let $M_{1},\ldots,M_{N-1}$ be a block of invertible matrices and set $\tilde{M}_n=M_{N-n}^{-1}$
    so that we obtain a new block $\tilde{M}_1,\ldots,\tilde{M}_{N-1}$ of invertible matrices. 
    By assumption, there exist a sub-block and 
    near-identity perturbations $\target{\tilde{M}_m},\ldots,\target{\tilde{M}_{m'}}$ 
    such that 
    $$
    \frac{s_j(\target{\tilde{M}_{m'}}\cdots\target{\tilde{M}_m})}{s_{j+1}(\target{\tilde{M}_{m'}}\cdots\target{\tilde{M}_m})}>\gap.
    $$
    Now let $\target{M_l}=(\target{\tilde{M}}_{N-l})^{-1}$ for $l=N-m',\ldots,N-m$. These
    are near-identity perturbations of the matrices $M_{N-m'},\ldots,M_{N-m}$. 
    Since $\target{M_{N-m}}\cdots\target{M_{N-m'}}=(\target{\tilde{M}_{m'}}\cdots\target{\tilde{M}_m})^{-1}$,
    we see $s_{d-j}(\target{M_{N-m}}\cdots\target{M_{N-m'}})=s_{j+1}(\target{\tilde{M}_{m'}}\cdots\target{\tilde{M}_m})^{-1}$
    and $s_{d-j+1}(\target{M_{N-m}}\cdots\target{M_{N-m'}})=s_{j}(\target{\tilde{M}_{m'}}\cdots\target{\tilde{M}_m})^{-1}$.
    This ensures that 
$$\frac{s_{d-j}(\target{M_{N-m}}\cdots\target{M_{N-n}})}{s_{d-j+1}(\target{M_{N-m}}\cdots\target{M_{N-m'}})}
    =\frac{s_{j}(\target{\tilde{M}_{m'}}\cdots\target{\tilde{M}_m})}{s_{j+1}(\target{\tilde{M}_{m'}}\cdots\target{\tilde{M}_m})}>\gap$$ as required.
\end{proof}

\subsection{Target when singular vectors remain spread}\label{S:far}

% \aq{Is $j$ fixed in this definition? Should we call it 
% $(\eta,j)$-spread? or is that too ugly?}
% \cgt{My interpretation was that $j$ was not fixed. I've tried to incorporate this now, keeping  $\delta_{n,j}$, as  we use it later. I don't think we need $(\eta,j)$-spread...}
% \aq{All we use in L5.3 is that the block is $(\eta,j)$-spread for what it's worth}

\begin{lemma}[Target for spread block]\label{lem:far}
Let $1\le j<d$, $\gap>1$, $0<\eta<1$ and let $0<\ep<1$.
%($\ep^ \gamma=\eta$). 
Consider a  product of 
$\blocklength-1\geq \frac{16 \log \gap}{\ep \eta}$ invertible matrices, $\mr{M_{\blocklength-1}}
\cdots\mr{M_{1}}$ (a ``block'').
%where each matrix is of norm at most $1-\frac\ep2$
%and has $d$th singular value at least $\frac\ep2$. 
If the block is $\eta$-spread, then there exists a product 
$\target{M_{\blocklength-1}},\ldots,\target{M_{1}}$ (a ``target block'')
with the properties:
\begin{enumerate}
    \item $\target{M_{n}}=R_n \mr{M_{n}}$ with $\|R_n^{\pm 1}-I\|
    \le\frac\ep 4$ for each $1\le n< \blocklength$;
    \item 
    $$
    \frac{s_j(\target{M_{\blocklength-1}}\cdots\target{M_{1}})}
    {s_{j+1}(\target{M_{\blocklength-1}}\cdots\target{M_{1}})}\ge \reqgap.
    $$
\end{enumerate}
\end{lemma}

\begin{proof}
    We first describe the simple idea: 
    after applying each $\mr{M_n}$, we expand the images of the leading $j$ singular
    vectors by a factor of $1+\frac {\ep\eta}8$.
 %   and contract the remaining images by a factor of $\exp(-\ep^{\rgapindex+1}/(5\sqrt d))$.
    The spread condition implies that the fast and slow subspaces are far enough apart that the
    perturbation to each matrix is of size at most $\frac\ep 4$. In $\blocklength-1$ steps,
    the cumulative effect of this is to increase the ratio between the $j$th and $(j+1)$st
    singular values by a factor of $\gap$, by the choice of $N$. 

Let the spaces $E_n$ and $F_n$ be as above so that $\RR^d=E_n\oplus F_n$.
By Lemma \ref{lem:projbound}, we have 
$\|\Pi_{E_n\parallel F_n}\|\le 2/\delta_n<2/\eta$.

We then set $R_n=I+\tfrac{\ep\eta}8\,\Pi_{E_{n}\parallel F_{n}}$. 
The calculation above shows that $\|R_n-I\|
=\|\tfrac{\ep \eta}8\,\Pi_{E_n\parallel F_n}\|
\le \tfrac{\ep \eta}8(2/\eta)\le\tfrac\ep 4$.
Similarly $R_n^{-1}=I-\big(\tfrac{\ep\eta}8/(1+\tfrac{\ep\eta}8)\big)\Pi_{E_n\parallel F_n}$
so that $\|R_n^{-1}-I\|\le\tfrac\ep 4$ also.
Since by definition, $M_n(E_{n-1})=E_{n}$ and $M_n(F_{n-1})=F_{n}$, we have
\begin{align*}
\mr{M_n}\Pi_{E_{n-1}\parallel F_{n-1}}&=\Pi_{E_{n}\parallel F_{n}}\mr{M_n}\quad\text{and}\\
\mr{M_n}\Pi_{F_{n-1}\parallel E_{n-1}}&=\Pi_{F_{n}\parallel E_{n}}\mr{M_n},\
\end{align*}
so that
\begin{align*}
\mr{M_n}\big(I+\tfrac{\ep \eta}8
\Pi_{E_{n-1}\parallel F_{n-1}}\big)&=
\mr{M_n}\big(\Pi_{F_{n-1}\parallel E_{n-1}}+(1+\tfrac{\ep \eta}8)
\Pi_{E_{n-1}\parallel F_{n-1}}\big)\\
&=\big(\Pi_{F_{n}\parallel E_{n}}+(1+\tfrac{\ep \eta}8)
\Pi_{E_{n}\parallel F_{n}}\big)\mr{M_n}.
\end{align*}
Using this inductively, along with the fact that $\Pi_{E_n\parallel F_n}$ and $\Pi_{F_n\parallel E_n}$ 
are complementary projections we see that
\begin{align*}
\target{M_{\blocklength-1}}\cdots\target{M_1}
&=\big(\Pi_{F_{\blocklength-1}\parallel E_{\blocklength-1}}+(1+\tfrac{\ep \eta}8)^{\blocklength-1}\Pi_{E_{\blocklength-1}\parallel F_{\blocklength-1}}\big)
\mr{M_{\blocklength-1}}\cdots\mr{M_1}.
\end{align*}
Hence 
\begin{equation*}
    s_i(\target{M_{\blocklength-1}}\cdots\target{M_1})=
    \begin{cases}
    (1+\tfrac{\ep \eta}8)^{\blocklength-1}s_i(\mr{M_{\blocklength-1}}\cdots\mr{M_1})&\text{for $i\le j$;}\\
    s_i(\mr{M_{\blocklength-1}}\cdots\mr{M_1})&\text{for $i>j$.}
    \end{cases}
\end{equation*}
In particular, 
\begin{align*}
\frac{s_j(\target{M_{\blocklength-1}}\cdots\target{M_1})}{s_{j+1}(\target{M_{\blocklength-1}}\cdots\target{M_1})}
&=
\left(1+\frac{\ep \eta}8\right)^{\blocklength-1}
\frac{s_j(\mr{M_{\blocklength-1}}\cdots\mr{M_1})}{s_{j+1}(\mr{M_{\blocklength-1}}\cdots\mr{M_1})}
\ge \left(1+\frac{\ep \eta}8\right)^{\blocklength-1}.
\end{align*}
Since $0<\ep<1$ and $0<\eta<1$, we have $1+\frac{\ep\eta}8>\exp(\frac{\ep\eta}{16})$.
Hence the choice of $N$ ensures that $\frac{s_j(\target{M_{\blocklength-1}}\cdots\target{M_{1}})}
    {s_{j+1}(\target{M_{\blocklength-1}}\cdots\target{M_{1}})}\ge \reqgap$, as required.
\end{proof}

\subsection{Targets  for $d=2$ and extremal singular value gaps}\label{S:2dtarget}
This section describes how to identify targets  for blocks of $2\times 2$ matrices. 
We start by presenting an auxiliary result about extremal singular value gaps
(that is $s_1(\cdot)/s_d(\cdot)$), 
valid for arbitrary dimension. The main result of this section is Lemma~\ref{lem:tard} which, specialised to $d=2$, yields targets for $2\times 2$ matrices.

\begin{lemma}\label{lem:closed2}
Let $B$ be a non-singular $d\times d$ matrix and suppose there exist
orthogonal vectors $u$ and $v$ such that $\angle(Bu,Bv)\leq \eta$. 
Then $s_1(B)/s_d(B)\geq 1/\eta$. 

The same conclusion holds if there exist $u$ and $v$ with $\angle(u,v)\leq\eta$
but $Bu$ and $Bv$ are orthogonal. 
\end{lemma}

\begin{proof}
We assume without loss of generality that $\|Bu\|\ge \|Bv\|$. 
We then have $\|B^{\wedge 2}(u\wedge v)\|=\|Bu\wedge Bv\|
\le \|Bu\|\|Bv\|\angle(Bu,Bv)\le\|Bu\|^2\eta$. 
Letting $L$ be the restriction of $B$ to $\lin(u,v)$
and $L^{\wedge 2}$ be its exterior square, it follows that
$\|L^{\wedge 2}\|\le \eta s_1(L)^2$. Since $\|L^{\wedge 2}\|
=s_1(L)s_2(L)$, it follows that $s_2(L)\le \eta s_1(L)$. 
Finally $s_1(B)\ge s_1(L)$ and $s_d(B)\le s_2(L)$
giving the required conclusion.
The second statement follows from taking inverses of $B$. 
    \end{proof}

\begin{lemma}[Target for $d\times d$ matrices]\label{lem:tard}
Let $0<\ep<1$ and $\gap>1$ be given.
Consider a block of ${\blocklength-1}\geq \frac{16}{\ep} \gap \log \gap$ 
$d\times d$ invertible matrices, $\mr{M_1},\ldots, \mr{M_{\blocklength-1}}$.
Then there exist $1\le m\leq m'<\blocklength$ and a sequence 
$\target{M_m},\ldots,\target{M_{m'}}$ of near-identity perturbations
of $M_m,\ldots,M_{m'}$
with the property
    $$
\frac{s_1(\target{M_{m'}}\cdots\target{M_m})}
    {s_{d}(\target{M_{m'}}\cdots\target{M_m})}\ge \reqgap.
    $$
\end{lemma}

\begin{proof}
We consider two cases:
\begin{enumerate}
    \item The block $\mr{M}_{\blocklength-1}, \dots, \mr{M}_1$ is $1/\gap$-spread. That is, $\delta_n>1/\gap$ for each $1\le n<\blocklength$.
    \item The block $\mr{M}_{\blocklength-1}, \dots, \mr{M}_1$ is nearly aligned. That is, there exists $1\le n<\blocklength$ such that $\delta_n\leq 1/\gap$.
\end{enumerate}
In the first case, applying Lemma~\ref{lem:far}, with $j=1$ and
$\eta=1/\gap$, identifies a target  for $\mr{M}_{\blocklength-1}, \dots \mr{M}_1$; namely 
$\target{B}:=\target{M_{\blocklength-1}}\dots \target{M_{1}}$, as constructed in 
Lemma~\ref{lem:far}, satisfies $\frac{s_1(\target B)}{s_d(\target B)}\ge
\frac{s_1(\target{B})}{s_2(\target{B})}\geq \reqgap$.
% Since $\target{A_n}=R_nA_n$ with $\|R_n-I\|<\frac\ep 4$, we see
% $\|\target{A_n}-A_n\|\le \|R_n-I\|\|A_n\|<\frac\ep 4$, so that 
% $\|\target{A_n}-A_n^0\|\le \frac{3\ep}4$, $\|\target{A_n}\|\le 1-\frac\ep 4$
% and $s_d(\target{A_n})\ge s_d(A_n)-\|\target{A_n}-A_n\|\ge \frac \ep 4$.

In the second case, the result follows from Lemma~\ref{lem:closed2}, with
$\target{B}:=\mr{M_{n}} \dots \mr{M_1}$.
\end{proof}

\begin{remark}\label{rem:tar2}
    Note that Lemma~\ref{lem:tard}, specialised to $d=2$, yields targets for $2\times 2$ matrices.
\end{remark}

\subsection{Targets for $d=3$}\label{S:3dtarget}
%\aq{I moved the definition of $E_{n,j}$, $F_{n,j}$ and $\delta_{n,j}$ and spread to the beginning of \S5.}
In this section, we describe how to identify targets for blocks of $3\times 3$ matrices.
We start with two auxiliary lemmas. The main result of this section is Lemma~\ref{lem:tar3}.
The reader should recall the definitions of $v_j$, $v^n_j$, $E^n_j$, $F^n_j$ and $\delta_{n,j}$ from the beginning of 
Section \ref{S: find target}. 
%\aq{Is $v^n_j$ defined somewhere?}
%\cgt{added $v^n_i$ to the preface}

\begin{lemma}\label{lem:12close}
    Let $M_1,\ldots,M_{\blocklength-1}$ be invertible $3\times 3$ matrices
    and let $B=M_{\blocklength-1}\cdots M_1$. Suppose $s_1(B)/s_2(B)< \gap$. 
    Suppose further that for some $1\le n< \blocklength$, $\angle(v^n_1,v^n_2)<\gap^{-5}$.
    Then, setting $B_1=M_n\cdots M_1$ and $B_2=M_{\blocklength-1}\cdots M_{n+1}$, either
    $s_1(B_1)/s_2(B_1)\ge \gap$ or
    $s_1(B_2)/s_2(B_2)\ge\gap$. %(which implies $s_1(B_1)/s_2(B_1)\ge \gap$ if $n=N$). 
\end{lemma}

\begin{proof}
We use the (known) inequality $s_1(B_2B_1)\ge s_{1+k}(B_2)s_{d-k}(B_1)$, valid in all dimensions.
We briefly indicate the proof: by the max-min characterization of singular values, 
namely  $s_k(B)=\max_{\{W:\mathrm{dim} W=k\}}\min_{\{w\in W, \|w\|=1\}}\|Bw\|$,
there is a $(d-k)$-dimensional subspace $U$ of $\RR^d$ such that $\|B_1u\|\ge s_{d-k}(B_1)$
for all $u\in U\cap S$ and a $(1+k)$-dimensional subspace $V$ of $\RR^d$ such that $\|B_2v\|\ge
s_{1+k}(B_2)$ for all $v\in V\cap S$. Since $\dim(B_1U)+\dim V=\dim U+\dim V>d$, the
spaces $B_1U$ and $V$ must intersect. That is, there exists $u\in U\cap S$ such that $B_1u\in V$. 
Hence $\|Bu\|=\|B_2(B_1u)\|\ge s_{1+k}(B_2)\|B_1u\|\ge s_{1+k}(B_2)s_{d-k}(B_1)$ as required.

% \aq{reference?}\cgt{Gary, is this in \cite{HornJohnson}?}
% \aq{There is a close cousin which is (7.3.14) in an exercise in Horn + Johnson 2nd Ed. Maybe we can find a better ref? 
% Or it follows pretty immediately from the minmax/maxmin characterizations of the svals.}
% \aq{I wrote out the maxmin proof if we want to include a proof.}

Specializing to the case $d=3$ and $k=1$ gives the inequality
\begin{equation}\label{eq:3dSVrel}
s_1(B)\ge s_2(B_2)s_2(B_1),
\end{equation}
valid for $3\times 3$ matrices.
Combining this with the assumption
that $B$ does not have a $\gap$-large (1,2)-gap, we see $s_2(B)\ge \gap^{-1}s_2(B_2)s_2(B_1)$
so that $S_1^2(B)\ge \gap^{-1}s_2(B_2)^2s_2(B_1)^2$. 

On the other hand, we have 
\begin{align*}
    S_1^2(B)&=s_1(B)s_2(B)\\
    &=\|Bv_1\wedge Bv_2\|\\
    &\le s_1(B_2)^2 \|B_1v_1\wedge B_1v_2\|\\
    &\le s_1(B_2)^2  \|B_1v_1\|\,\|B_1v_2\|\gap^{-5}\\
    &\le s_1(B_2)^2s_1(B_1)^2\gap^{-5}.
\end{align*}

Combining this with the preceding inequality gives
$$
\frac{s_1(B_2)^2s_1(B_1)^2}
{s_2(B_2)^2s_2(B_1)^2}\ge \gap^{4},
$$
so that, by taking square roots, at least one of $s_1(B_2)/s_2(B_2)\ge\gap$
and $s_1(B_1)/s_2(B_1)\ge\gap$ must hold as required.
\end{proof}

\begin{lemma}\label{lem:build12gap13close}
Let $0<\ep<1$, $\gap>2$ and $N>2$.
%$N=\lceil \ep^{-7\rgapindex}\rceil$
Let $\mr{M_1},\ldots, \mr{M_{\blocklength-1}}$ be a sequence of invertible $3\times 3$ matrices.
%of norm at most $1-\frac\ep2$ such that $s_3(\mr{A_i})\ge \frac\ep4$ 
%for each $1\le i\le N$. 
%\aq{removed conditions on norms of the $\mr{A_i}$ as they are not needed in the new reformulation and don't work with the complementarity lemma.}
% Let $v_1,\ldots,v_d$ be singular vectors of $B=A_{\blocklength-1}\cdots A_1$,
% listed  in order of decreasing singular values. 
% For each $1\le i\le 3; \ 1\le n<\blocklength$,
% let $v_i^n=\frac{A_n\cdots A_1 v_i}{\|A_n\cdots A_1 v_i\|}$. 
If there exists $1\le n\le N-1$ such that $\delta_{n,2}<\min\{\gap^{-9},\frac\ep4\}$,
%Let $B_1=A_n\cdots A_1$, $B_2=A_{\blocklength-1}\cdots A_{n+1}$
% and suppose that there exists 
% $w^n=B_1w^0\in E_{n,2}= \lin(v_1^n, v_2^n)$ such that $\angle(w^n,v_3^n)< \min\{\gap^{-9},\frac\ep4\}$.
% $j\in \{1,2\}$ such that $\angle(v_j^n,v_3^n)< \gap^{-9}$.
%Furthermore, assume 
%$s_1(B)/s_2(B),s_1(B_1)/s_2(B_1), s_1(B_2)/s_2(B_2)< \ep^{-\rgapindex}$ .
then there exist $1\le m\leq m'< N$, a sequence of orthogonal matrices $R_m,\ldots,R_{m'}$ and target matrices
$\target{M_i}=R_i\mr{M_i}$
such that 
\begin{enumerate}
    \item $\|R_i^{\pm 1}-I\|\le\frac\ep 4$ for each $m\le i\le m'$;
    \label{it:smallpertvjv3}
    \item \label{it:biggapvjv3}
    $$
    \frac {s_1(\target{M_{m'}}\cdots\target{M_m})}{s_2(\target{M_{m'}}\cdots\target{M_m})}
    \ge \gap.
    $$
\end{enumerate}
\end{lemma} 

\begin{proof}
Let $B=M_{N-1}\cdots M_1$, $B_1=M_n\cdots M_1$, and $B_2=M_{N-1}\cdots M_{n+1}$.
If $s_1(B)/s_2(B)\geq \gap$, $s_1(B_1)/s_2(B_1)\geq \gap$ or $s_1(B_2)/s_2(B_2)\geq \gap$, 
we can use $B$, $B_1$ or $B_2$ as the target block respectively (with $m=1$, $m'=N-1$; 
$m=1$, $m'=n$; or $m=n+1$, $m'=N-1$, 
respectively, and $R_i=I$ for all $i$) to satisfy the conclusion of the lemma. 

Hence we assume $B, B_1$ and $B_2$ have no $\gap$-large $(1,2)$-gap. That is, we assume $s_1(B)/s_2(B)$, 
$s_1(B_1)/s_2(B_1)$ and $s_1(B_2)/s_2(B_2)$ are all less than $\gap$.
Using \eqref{eq:3dSVrel} and the absence of a $\gap$-large
(1,2)-gap for $B_1$ and $B_2$, we see
\begin{equation}
    s_1(B) \geq s_2(B_1) s_2(B_2) \geq \gap^{-2} s_1(B_1) s_1(B_2). \label{eq:vjv3SVCompBounds1}
\end{equation}
Since $B$ does not have a $\gap$-large (1,2)-gap, we see from \eqref{eq:3dSVrel}
\begin{equation}
  s_2(B)\geq \gap^{-1} s_1(B)\ge \gap^{-1}s_2(B_1) s_2(B_2).\label{eq:vjv3SVCompBounds2}\\
\end{equation}
Finally, using the fact that $S_1^3(B)=|\det B|=|\det B_1||\det B_2|=S_1^3(B_1)S_1^3(B_2)$
together with \eqref{eq:vjv3SVCompBounds1} and \eqref{eq:vjv3SVCompBounds2} we see
\begin{equation}
    s_3(B)\leq \gap^{3} s_3(B_1) s_3(B_2).
\label{eq:vjv3SVCompBounds3}
\end{equation}

By hypothesis, there exist $e\in E_{n,2}\cap S$ and $f\in F_{n,2}\cap S$ 
satisfying $\|e-f\|<\min\{\gap^{-9},\frac\ep4\}$. 
%\aq{This contradicts the defn of $v_3^n$ at the start
%of the section}
%Without loss of generality, we take $w^n$ be as in the statement of the lemma to be normalized. 
We let $m=1$, $m'=N-1$ and let $R_i=I$ for each $i\ne n$ so that $\target{M_i}=M_i$, for each $i\neq n$.
We define $\target{M_n}=R_nM_n$, 
where $R_n$ is a near identity orthogonal transformation 
%rotating the plane $\lin(v_j^n,v_3^n)$ 
such that $R_n e=f$.
%$R_n v_j^n \in \lin(v_3^n)$.
% By a choice of sign, we may assume that 
% $R_n w^n=v_3^n$. 
%$R_nv_j^n=v_3^n$. 
Since $\|f-e\|<  \min\{\gap^{-9},\frac\ep4\}$,
%$\angle(v_j^n,v_3^n)< \ep^{\vtvjcloseexp \rgapindex}$, 
$R_n$ may be chosen so that
$\|R_n^{\pm 1}-I\| \leq \frac{\ep}{4}$.
We next show that $\target{B}=B_2 R_n B_1$ satisfies condition \eqref{it:biggapvjv3}.
% Relying on (i) \eqref{eq:3dSVrel} and the absence of gap for $B_1, B_2$ for \eqref{eq:vjv3SVCompBounds1}, (ii) the absence of gap for $B$ and \eqref{eq:vjv3SVCompBounds1} for \eqref{eq:vjv3SVCompBounds2}, and (iii) the fact that $S_1^3(B)=\det(B)=\det(B_1)\det(B_2)=S_1^3(B_1)S_1^3(B_2)$, together with \eqref{eq:vjv3SVCompBounds1} and \eqref{eq:vjv3SVCompBounds2} for \eqref{eq:vjv3SVCompBounds3}, we have 
% \begin{eqnarray}%\label{eq:vjv3SVCompBounds}
%     %\begin{split}
% s_1(B) \geq s_2(B_1) s_2(B_2) \geq \gap^{-2} s_1(B_1) s_1(B_2), \label{eq:vjv3SVCompBounds1}\\
% s_2(B)\geq \gap^{-1} s_1(B)\ge \gap^{-1}s_2(B_1) s_2(B_2),
% \label{eq:vjv3SVCompBounds2}\\
% s_3(B)\leq \gap^{3} s_3(B_1) s_3(B_2).
% \label{eq:vjv3SVCompBounds3}
%     %\end{split}
% \end{eqnarray}
Let $e_0$ and $f_0$ be such that $B_1e_0=e$ and $B_1f_0=f$. 

Since $f_0\in F_2(B)$, $\|Bf_0\|=s_3(B)\|f_0\|$.
By properties of singular values, $\|B_1f_0\|\ge s_3(B_1)\|f_0\|$. 
Hence
    %$\|B_2 R_n B_1 v_0\| =\|B_1 v^0_j\|
    $\|B_2 R_n B_1 \frac{e_0}{\|e_0\|}\| =\|B_1 \frac{e_0}{\|e_0\|}\|
    \frac{\|Bf_0\|}{\|B_1 f_0\|}
    \leq s_1(B_1) \frac{s_3(B)}{s_3(B_1)} 
    \leq s_1(B_1) s_3(B_2) \gap^{3}$,
  where the last inequality follows from  \eqref{eq:vjv3SVCompBounds3}.
  Lemma~\ref{lem:closed2} ensures that 
$s_1(B_2)/s_3(B_2)\geq \gap^{9}$. Thus, using 
\eqref{eq:vjv3SVCompBounds1} in the last step, we get
  \begin{equation}\label{eq:bdv_jgrowth}
      %\|B_2 R_n B_1 f^0_j\|\leq
       \|B_2 R_n B_1 \tfrac{e_0}{\|e_0\|}\|\leq
      \gap^{-6} s_1(B_1) s_1(B_2)  \leq \gap^{-4} s_1(B).
  \end{equation}

Next, $\|B_2 R_n B_1 \frac{f_0}{\|f_0\|}\| \leq s_1(B_2) \|B_1 \frac{f_0}{\|f_0\|}\|
\leq 
s_1(B_2) \frac{\|B f_0\|/\|f_0\|}{s_3(B_2)}=s_1(B_2) \frac{s_3(B)}{s_3(B_2)}
\leq s_1(B_2) s_3(B_1) \gap^{3},$
where the last inequality follows from  \eqref{eq:vjv3SVCompBounds3}.
  Lemma~\ref{lem:closed2} ensures that $s_1(B_1)/s_3(B_1)\geq \gap^{9}$. Thus,
  using \eqref{eq:vjv3SVCompBounds1} again, we have
    \begin{equation}\label{eq:bdv_3growth}
      \left\|B_2 R_n B_1 \frac{f_0}{\|f_0\|}\right\| \leq s_1(B_2) s_1(B_1) \gap^{-6} \leq s_1(B) \gap^{-4}.
  \end{equation}
Note that, by construction, $e_0 \perp f_0$.
%$v_j^0 \perp v_3^0$. 
Then, \eqref{eq:bdv_jgrowth} and \eqref{eq:bdv_3growth} imply that for every  %$0\neq v\in \lin(e_0,f_0)$, 
$0\neq v\in \lin(e_0,f_0)$, 
$\|B_2 R_n B_1 v\|/\|v\| \leq \sqrt{2} 
\gap^{-4}s_1(B)
\leq \gap^{-3}s_1(B)$. 
Therefore by the min-max characterization of singular values,
\begin{equation}    \label{eq:bds2B2RB1}
s_2(B_2 R_n B_1)\leq \gap^{-3}s_1(B).
\end{equation}

In view of 
\eqref{eq:3dSVrel},
\begin{equation} \label{eq:bds1B2RB1}
        s_1(B_2 R_n B_1) \geq s_2(B_2 R_n) s_2(B_1)=s_2(B_2) s_2(B_1) \geq \gap^{-2} s_1(B_2)  s_1(B_1) \geq \gap^{-2} s_1(B).
\end{equation}
Since $\target{B}=B_2 R_n B_1$,  \eqref{eq:bds2B2RB1} and \eqref{eq:bds1B2RB1} imply 
$s_1(\target{B})/s_2(\target{B})\geq \gap$, as required.
\end{proof}

\begin{lemma}[Target for $3\times 3$ matrices]\label{lem:tar3}
Let $\gap>2$, $0<\ep<1$, $\eta=\min\{\gap^{-9},\frac{\ep}{4}\}$ and consider a product of  $3\times 3$ invertible matrices, $\mr{M_{\blocklength-1}},\ldots, \mr{M_{1}}$, with $\blocklength-1\geq \frac{16}{\ep \eta^2}  \log \gap$.
%where each matrix is of norm at most $1-\frac\ep2$
%and all its singular values are bounded below by $\frac\ep2$. 
% If $\ep$ is sufficiently small,
% \aq{The way this is written makes me nervous. It sounds as though how close $\ep$ has to be to 0 
% depends on the block... I think we should take care to show that the choice of 
% $\ep$ is universal.} 
% \aq{I think the only place that small $\ep$ showed up was in L5.3 and it's now removed.}
Then there exists $1\le m\leq m'<N$ and a sequence 
$\target{M_{m}},\ldots,\target{M_{m'}}$
with the properties:
\begin{enumerate}
    \item $\target{M_{n}}$ is a near-identity perturbation of $M_n$ for each $m\le n\le m'$;
    % =R_n \mr{M_{n}}$ with $\|R_n^{\pm1}-I\|
    % \le\frac\ep 4$ for each $1\le n< \blocklength$;
     \item The target $\target{B}:=\target{M_{m'}}\cdots\target{M_{m}}$ satisfies
    $$
\frac{s_1(\target{B})}
    {s_{2}(\target{B})}\ge \reqgap.
    $$    
\end{enumerate}
\end{lemma}

\begin{proof}
For each $1\leq j<3$ and $1\leq n< \blocklength$ we recall that $\delta_{n,j}$ denotes the minimal distance between 
points in $E_{j,n}\cap S$ and $F_{j,n}\cap S$ and $\delta_n=\min_j\delta_{n,j}$.
%Let $\eta=\min\{\gap^{-9},\ep/4\}$. 

We consider three partially overlapping cases that cover all situations:
\begin{enumerate}
    \item 
    %The block $\mr{A}_{\blocklength-1}, \dots, \mr{A}_1$ is nearly aligned, and 
    There exists $1\le n<\blocklength$ such that $\delta_{n,2}< \eta$. %\cgt{need $\eta<\ep^{\vtvjcloseexp\rgapindex }$ here.}
    \item 
    %The block $\mr{A}_{\blocklength-1}, \dots, \mr{A}_1$ is nearly aligned, but 
    There exists $1\le n<\blocklength$ such that $\delta_{n,1}\leq \eta^2$, and $\delta_{n,2}\ge \eta$ for every $1\leq n < \blocklength$.
    \item The block $\mr{M}_{\blocklength-1}, \dots, \mr{M}_1$ is $\eta^2$-spread. That is, $\delta_n>\eta^2$ for each $1\le n<\blocklength$. %\cgt{may need to change the current $\eta^2=\ep^\rgapindex$ to a much smaller threshold in Lemma~\ref{lem:far}, e.g. $\eta=\gap^{-9}$}
\end{enumerate}
In the first and third cases, Lemma~\ref{lem:build12gap13close} and Lemma~\ref{lem:far}, respectively, 
identify targets for $\mr{M}_{\blocklength-1}, \dots \mr{M}_1$, namely $\target{B}:=\target{M_{\blocklength-1}}\dots \target{M_{1}}$ satisfies $\frac{s_1(\target{B})}{s_2(\target{B})}\geq \reqgap$. 
%\footnote{Note that only the gap $\angle(E_{n,1}, F_{n,1})$ is relevant to build a gap between $s_1$ and $s_2$, but it is helpful for the proof to consider $\angle(E_{n,2}, F_{n,2})$ as well.} 

To finish, we show how to identify a target block in the second case.
Let $1\leq n \leq \blocklength-1$ be such that
$\delta_{n,1}\leq \eta^2$, and let $\alpha v_2^n+\beta v_3^n\in \lin(v_2^n, v_3^n)\cap S$ be such that 
$\|v_1^n+\alpha v_2^n+\beta v_3^n\|\leq \eta^2$.
We wish to show $|\beta|\le 2\eta$.
If $\beta=0$, this is immediate. If $\beta\ne 0$, then we have $d(v_3^n,\lin(v_1^n,v_2^n))\le\eta^2/|\beta|$,
so using Lemma \ref{lem:linspacebound}, $\delta_{n,2}\le 2\eta^2/|\beta|$.
By assumption, $\delta_{n,2}\ge\eta$, so it follows that $|\beta|\le 2\eta$.

% Note that $|\beta|\leq 2\eta$ because otherwise, 
% $\delta_{n,2}\leq 2d(v_3^n, \lin(v_1^n, v_2^n)) \leq 2\beta^{-1}\|v_1^n+\alpha v_2^n+\beta v_3^n\| \leq \eta$, which yields a contradiction.

Hence, $\|v_1^n+\alpha v_2^n\|\leq 2\eta+\eta^2$.
Therefore, using Lemma~\ref{lem:linspacebound}, $\angle(v^n_1,v^n_2)\leq 2(2\eta+\eta^2)< \gap^{-5}$.
Thus Lemma~\ref{lem:12close} shows that either 
$\target{B}=M_{N-1}\cdots M_1$ or
$\target{B}=B_1=M_n\cdots M_1$ or $\target{B}=B_2=M_{\blocklength-1}\cdots M_{n+1}$ yield a
target for $\mr{M}_{\blocklength-1}, \dots \mr{M}_1$. %That is, $\target{A}$ satisfies $\frac{s_1(\target{A})}{s_2(\target{A})}\geq \reqgap$.

\end{proof}

\begin{corollary}\label{cor:tar3}
Let $\gap>2$, $0<\ep<1$ and $\eta=\min\{\gap^{-9},\frac\ep4\}$.
Consider a product of  $3\times 3$ invertible matrices, $\mr{M_{\blocklength-1}},\ldots, \mr{M_{1}}$, with $\blocklength-1\geq \frac{16}{\ep \eta^2}  \log \gap$.
%     Consider a product of $\blocklength\geq \frac{16 \log \gap}{
% \ep^{\vtvjcloseexp\rgapindex +1}}$ $3\times 3$ invertible matrices, $\mr{A_{\blocklength}},\ldots, \mr{A_{1}}$.
% where each matrix is of norm at most $1-\frac\ep2$
% and all its singular values are bounded below by $\frac\ep2$. 
Then, there exist $1\le m\leq m'<N$ and a sequence 
$\target{M_m},\ldots,\target{M_{m'}}$ with the properties:
\begin{enumerate}
    % \item $\target{A_{n}}$=R_n \mr{A_{n}}$ with $\|R_n^{\pm1}-I\|
    % \le\frac\ep 4$ for each $1\le n< \blocklength$;
    \item $\target{M_n}$ is a near-identity perturbation of $M_n$ for each $m\le n\le m'$;
    \item The target $\target{B}:=\target{M_{m'}}\cdots\target{M_{m}}$ satisfies
    $$
\frac{s_2(\target{B})}
    {s_{3}(\target{B})}\ge \reqgap.
    $$    
\end{enumerate}
\end{corollary}
\begin{proof}
The result follows directly from Lemma~\ref{lem:tar3} and  Lemma~\ref{lem:complementarity}.
% \cgt{I don't think this is true as written, as the complementarity result requires chequing something for *all* invertible matrices, and in the previous lemma we have (and I think use) the restriction $\|A_n\|< 1$.}
% \aq{I don't think we used any facts about the norms in the proof of L5.9? So I think that could be removed from the statement} 
\end{proof}

\section{Abstract Targets}\label{S:abstract}
In this section, we show that one may find targets in all dimensions $d$, and for each pair $j,j+1$ of exponents 
with $1\le j\le d-1$. The proof shows that for all $0<\ep<1$ and all $\gap>0$, there exists an $N$
such that the target property is satisfied. However we do not give any upper bounds for $N$. The proof is
based on a contradiction argument. One supposes for a contradiction that for some $0<\ep<1$ and some $\gap>0$
there is no such $N$.
Then  there would be arbitrarily
long blocks of matrices for which no perturbation has $s_j(\cdot)/s_{j+1}(\cdot)$ exceeding $\gap$. 
One may then take a limit of these blocks to obtain a closed shift-invariant collection of sequences
of matrices with the property that for any perturbation of any sub-block, the ratio $s_j(\cdot)/s_{j+1}(\cdot)$ 
remains below $\gap$. In particular, when perturbing such a matrix cocycle, 
one gets that $\mu_j^{\ep/(4d)}=\mu_{j+1}^{\ep/(4d)}$. This construction is reminiscent
of the Furstenberg Correspondence Principle \cite{FurstenbergBook}, used in Additive Combinatorics, to relate questions about
configurations in positive density subsets of the integers with questions in dynamical systems. 
It also has some resemblance to the argument in the paper \cite{Bochi-Gourmelon} of Bochi and Gourmelon 
where a dynamical system was constructed to relate two non-dynamical statements. 
The conclusion of Lemma \ref{lem:compact}, however, contradicts Theorem \ref{thm:simp} in which we show that additive noiselike
perturbations of cocycles always have simple Lyapunov spectrum. Unlike in the previous section, we give no information on how the targets
are constructed. Rather, we infer their existence from the contradiction and compactness argument described above.

In this section, we consider an invertible ergodic measure-preserving transformation, $\sigma$, of 
a probability space $(\Omega,\rho)$.
% \gf{Previously (everywhere except section 5) $\mathbb{P}$ described something related to what we call here $\mathbb{Q}$.  In this section we have the ``new'' measure to describe random unperturbed cocycles, which we didn't have before, so this probably should have a new notation.
% It would be good if we could at lest partially harmonise  the $\mathbb{Q}$ notation with the previous $\mathbb{P}$ notation since they describe related things.} \aq{Done, I think.}
% \aq{I have an attachment to using $(\Omega,\PP)$ as a measure-preserving system. I can cope with calling $\QQ$ the uniform measure
% on perturbed matrices (e.g. as in C3.3), but I'm not 100\% convinced this will end up being clearer.}
If $A\colon \Omega\to\text{Mat}_d(\RR)$ is a map, we write
$A^{(n)}(\om)=A(\sigma^{n-1}\omega)\cdots A(\omega)$. We let $S_\infty=\{M\in\text{Mat}_d(\RR)\colon
|M|_\infty\le 1\}$ and let $S_\infty$ be equipped with the uniform measure, so that each entry is 
uniformly distributed in the range $[-1,1]$ and distinct entries are independent. 
We also use the notation $S$ to denote $\{M\in\text{Mat}_d(\RR)\colon \|M\|\le 1\}$. 
As before, we denote by $\PP$ the measure on $S_\infty^\ZZ$ where the matrices in the sequence
are mutually independent and each is distributed uniformly as described above. 
Finally let $\bar\Omega=\Omega\times S_\infty^\ZZ$ and we equip $\bar\Omega$ with the  
measure $\rho\times \PP$, invariant under $\bar\sigma:=\sigma\times\mathsf{shift}$. 
Given $(\omega,\zeta)\in \Omega\times S_\infty^\ZZ$, we write
$A_\ep(\om,\zt)=A(\om)+\ep\zt_0$ (that is, we perturb $A(\om)$ by $\ep$ times the zeroth matrix in
the sequence $\zt$). We then write $A_\ep^{(n)}(\om,\zt)=A_\ep(\bar\sigma^{n-1}(\om,\zt))\cdots A_\ep(\om,\zt)$. 
Since $\rho\times\PP$ is ergodic, $\lim_{n\to\infty}\frac 1n\log s_j(A_\ep^{(n)}(\om))$ exists $\rho\times\PP$-a.e.\ and 
is almost surely equal to a constant that we call $\mu^\ep_j$. As usual, we let $\lambda^\ep_1>\ldots>\lambda^\ep_k$ be
the distinct almost-sure constant values taken by $\mu^\ep_j$ and we let $m^\ep_j$ be the multiplicity of $\lambda^\ep_j$.

\begin{lemma}[Compactness]\label{lem:compact}
Suppose that for some $0<\epsilon<1$, for all $\delta>0$, there exists a sequence
of $d\times d$ matrices $(A_i)$ of norm at most 1 and $1\le j<d$, such that 
$$
\liminf \frac 1n\log\frac{s_j(A_\ep^{(n)})}{s_{j+1}(A_\ep^{(n)})}<\delta\text{ a.s.}
$$
Then there exists an ergodic invariant measure on $S^\ZZ$
such that 
$\mu_j^{\ep/(4d)}=\mu_{j+1}^{\ep/(4d)}$.
\end{lemma}

\begin{proof}%[Sketch Proof]
Notice that the hypothesis of the lemma implies that the conclusion of Corollary 
\ref{cor:metareform} fails, so that the hypothesis of Corollary \ref{cor:metareform}
must fail also. 
Hence we deduce that there exists $\gap>0$ such that for all $N$, there exists a block $M_1,\ldots,M_N$ of matrices
with the properties:
\begin{itemize}
\item $\|M_i\|\le 1-\frac\ep 2$ for each $i$;
\item $s_d(M_i)\ge \frac\ep 2$ for each $i$;
\item for all $1\le m\leq m'\le N$ and all $M'_{m},\ldots,M'_{m'}$ such that
$|M'_i-M_i|_\infty\le \frac\ep{4d}$ for each $m\le i\le m'$, one has 
$$
\frac {s_j(M'_{m'}\cdots M'_m)}{s_{j+1}(M'_{m'}\cdots M'_m)}\le \gap.
$$
\end{itemize}
For each $N$, let $M_{N,1},\ldots,M_{N,2N+1}$
% \cgt{I didn't quite get the $2N+1$. Doesnt it work with  $M_{N,1}, ..., M_{N,N}$ only? I may be missing something...}
% \aq{I was trying to build something two-sided. This is important because we need the invertible version of Osel.}
% \cgt{I see, I guess something that was slightly confusing to me was that one is using a 2N+1 long block, but I see now, maybe we could add the next bit for clarity...}
be such a block with the above property, of length $2N+1$, and build an element $A_N$ of $S^\ZZ$ by
$A_N=(A_{N,i})_{i\in\ZZ}$ by 
$$
A_{N,i}=
\begin{cases} 
M_{N,i+N+1}&\text{if $-N\le i\le N$;}\\
I&\text{otherwise.}
\end{cases}
$$
Equip $S$ 
with the norm topology induced by $\|\cdot\|$ and $S^\ZZ$ with the product topology. Then
$S$ is compact and so is $S^\ZZ$. 
Let $A=(A_i)_{i\in\ZZ}$ be a sub-sequential limit of $A_N=(A_{N,i})_{i\in\ZZ}$.
That is, there is a sequence $(N_q)$ converging to $\infty$
such that $A_i=\lim_{q\to\infty}A_{N_q,i}$ for each $i$.
Then let $-\infty<m\leq m'<\infty$ and let $B_m,\ldots,B_{m'}$ be an arbitrary perturbation 
such that $|B_i-A_i|_\infty\le \frac\ep{4d}$ for each $i$.  
Let $B_{q,i}=A_{N_q,i}+(B_i-A_i)$ so that $B_{q,i}\to B_i$ for
each $m\le i\le m'$ and $|B_{q,i}-A_{N_q,i}|_\infty\le\frac\ep{4d}$ for each $m\le i\le m'$. 
By the choice of the $A_N$'s and the assumption, we have that for all sufficiently large $q$, 
$$
\frac{s_j(B_{q,m'}\cdots B_{q,m})}{s_{j+1}(B_{q,m'}\cdots B_{q,m})}\le \gap.
$$
By continuity of singular values, 
$$
\frac {s_j(B_{m'}\cdots B_m)}{s_{j+1}(B_{m'}\cdots B_m)}\le \gap.
$$
Let $S_0=\{M\in\text{Mat}_d(\RR)\colon \|M\|\le 1-\frac\ep 2\text{ and }s_d(M)\ge \frac\ep 2\}$ and set
$$
\Omega=\{A\in S_0^\ZZ\colon q_j(B_{m'}\cdots B_m)\le\gap 
\text{ for all $m<m'$ whenever }|B_i-A_i|_\infty\le \tfrac\ep {4d},\ \forall i\},
$$
%\cgt{, m, m'}\},
where $q_j(M)=s_j(M)/s_{j+1}(M)$ as before.
The arguments above show that $\Omega$ is a non-empty closed, and hence
compact, shift-invariant subset of $S_0^\ZZ$. By the Krylov-Bogoliubov theorem,
there exists a shift-invariant measure supported on $\Omega$. Taking an ergodic
component, we obtain an ergodic invariant measure supported on $\Omega$. 

Now for every $\om\in\Omega$ and every sequence $(\Xi_i)$ in $S_\infty^\ZZ$,
by definition 
$$
\frac{s_j\big((A(\sigma^{n-1}\omega)+\tfrac\ep {4d}\Xi_{n-1})\cdots (A(\om)+\tfrac\ep {4d}\Xi_0)\big)}
{s_{j+1}\big((A(\sigma^{n-1}\omega)+\tfrac\ep {4d}\Xi_{n-1})\cdots (A(\om)+\tfrac\ep {4d}\Xi_0)\big)}\le\gap.
$$
It follows that $\mu_j^{\ep/(4d)}=\mu_{j+1}^{\ep/(4d)}$.
%\gf{Update $\lambda$ here to $\mu$?}
\end{proof}

The following result is essentially well known (see for example \cite{Bochi}). We include the proof for completeness.

\begin{lemma}[(Semi-)continuity]\label{lem:semi-cont}
Let $\sigma$ be an invertible ergodic measure-preserving transformation of a probability space $(\Omega,\rho)$
and let $A\colon \Omega\to\mathrm{Mat}_m(\RR)$ be a cocycle of matrices such that $\|A(\om)\|^{-1}\le c$ a.e.\ 
and $\int\log\|A(\om)\|\,d\rho<\infty$. 
Suppose $A$ has trivial Lyapunov spectrum. That is,
there exists $\lambda$ such that $\lim_{n\to\infty}\frac 1n\log s_j(A^{(n)}(\om))=\lambda$ a.e.\ for
$j=1,\ldots,m$. Then for all $\ep>0$, there exists a $\delta>0$ with the following property:

Let $\bar\sigma\colon (\bar\Omega,\bar\rho)\to (\bar\Omega,\bar\rho)$ be an ergodic invertible extension of 
$\sigma\colon (\Omega,\rho)\to (\Omega,\rho)$ with factor map $\pi$.
% \gfadd{Let $(\bar\Omega,\bar \PP)$ be an extension of $(\Omega,\PP)$.}
% \aq{Not quite sure how to say this. I don't think you can really say that $(\bar\Omega,\bar\PP)$ is an
% extension of $(\Omega,\PP)$. You could say $(\bar\Omega,\bar\PP,\bar\sigma)$ is an extension
% of $(\Omega,\PP,\sigma)$ except that we haven't been using tuple language}
% \gf{I don't mind how it is phrased, but I thought we should say what $\bar{\Omega}$ and $\bar{\mathbb{P}}$ are in the Lemma statement;  I just copied text introducing $\bar{\Omega}$ and $\bar{\mathbb{P}}$ from the proof.}
% \aq{I see. I don't like the language in the proof either!! One possibility is:
% let $\bar\sigma\colon (\bar\Omega,\bar\PP)\to (\bar\Omega,\bar\PP)$ be extension of 
% $\sigma\colon (\Omega,\PP)\to (\Omega,\PP)$. It's a bit of a mouthful, but it avoids introducing tuples.}
%\gf{sounds good.}
% If $\bar\sigma$ is a ergodic invertible measure-preserving transformation of $(\bar\Omega,\bar\PP)$
% where $\bar\sigma$ is an extension of $\sigma$ with factor map $\pi$ and
If $B\colon\bar\Omega\to\text{Mat}_m(\RR)$ 
%\cgt{should $k=m$? or even $m=k=d$, for consistency?}
%\aq{Oops. $k$ and $m$ should be the same. I somewhat deliberately didn't use $d$ because this 
%gets applies to a single Oseledets subspace. Replaced $k$ by $m$ for now.} 
satisfies $\|B(\bar\om)-A(\pi(\bar\om))\|\le \delta$ 
for $\bar\rho$-a.e.\ $\bar\om$, then all Lyapunov exponents of the cocycle $B$ lie in the
interval $(\lambda-\ep,\lambda+\ep)$. 
\end{lemma}

The proof here is essentially a semi-continuity result (that holds for arbitrary cocycles), and the additional 
assumption that the unperturbed Lyapunov spectrum is trivial yields the continuity result. 

\begin{proof}
    By the Multiplicative ergodic theorem and the Kingman sub-additive ergodic theorem, 
    \begin{align*}
        \lambda&=\lim_{n\to\infty}\frac 1n\int\log\|A^{(n)}(\om)\|\,d\rho
        =\inf_n\frac 1n\int\log\|A^{(n)}(\om)\|\,d\rho\text{\quad and}\\
        -\lambda&=\lim_{n\to\infty}\frac 1n\int\log\|(A^{(n)}(\om))^{-1}\|\,d\rho
        =\inf_n \frac 1n\int\log\|(A^{(n)}(\om))^{-1}\|\,d\rho.
\end{align*}
    Let $n$ be such that 
    \begin{align*}
        &\frac 1n\int\log\|A^{(n)}(\om)\|\,d\rho<\lambda+\ep\text{ and}\\
        &\frac 1n\int\log\|(A^{(n)}(\om))^{-1}\|\,d\rho<-\lambda+\ep.
    \end{align*}
    Define 
    \begin{align*}
    f_\delta(\om)&=\max_{\|\Xi_0\|,\ldots,\|\Xi_{n-1}\|\le \delta}\log\Big\|(A(\sigma^{n-1}(\om))+\Xi_{n-1})\cdots
    (A(\om)+\Xi_0)\Big\|\text{ and}\\
    g_\delta(\om)&=\max_{\|\Xi_0\|,\ldots,\|\Xi_{n-1}\|\le \delta}\log\Big\|\big((A(\sigma^{n-1}(\om))+\Xi_{n-1})\cdots
    (A(\om)+\Xi_0)\big)^{-1}\Big\|.
    \end{align*}
    Let $\delta<\min(1,\frac 1{2c})$. For $0<\ep<\delta$ and $\Xi$ such that $\|\Xi\|\le 1$ and $A$ such that $s_d(A)\ge \frac 1c$,
    we have $(\frac 1c-\frac 1{2c})\|x\|\le \|(A+\ep\Xi)x\|\le ((\max(\|A\|,1)+1)\|x\|$.
    Since $\log(\max(t,1)+1)\le \log^+t+1$, the right inequality gives
    that  $\log\|(A+\ep\Xi)\|\le \log^+\|A\|+1$. The left inequality gives $\log\|(A+\ep\Xi)^{-1}\|\le \log(2c)$.
    This yields $f_\delta(\om)\le \sum_{i=0}^{n-1}(1+\log^+\|A(\sigma^i\om)\|)$
    and $g_\delta(\om)\le n\log {2c}$. Since $\|A\|\,\|A\|^{-1}\ge 1$, we also obtain
    $f_\delta(\om)+g_\delta(\om)\ge 0$, giving the bounds
    \begin{align*}
    &n\log\tfrac{1}{2c}\le f_\delta(\om)\le \sum_{i=0}^{n-1}(1+\log^+\|A(\sigma^i\om)\|)\text{ and}\\
    &-\sum_{i=0}^{n-1}(1+\log^+\|A(\sigma^i\om)\|)\le g_\delta(\om)\le  n\log(2c).
    \end{align*}

    % $n\log\frac{1}{2c}\le f_\delta(\om)\le \sum_{i=0}^{n-1}(1+\log^+\|A(\sigma^i\om)\|)$
    % and, $-\sum_{i=0}^{n-1}(1+\log\|A(\sigma^i\om)\|)\le g_\delta(\om)\le  n\log(2c)$.
    % \gf{Anthony, can you add a hint here about at least the inequalities for $f_\delta$? 
    % I see for the RHS you are probably using submultiplicitivity of norms and sum/product of logs, 
    % but after that it is not obvious to me what is used.} 
    %\aq{Gary: I added explanation and a small correction. Hopefully now correct and clearer.}
    Also $f_\delta(\om)
    \to\log\|A^{(n)}(\om)\|$ and $g_\delta(\om)\to \log\|(A^{(n)}(\om))^{-1}\|$ for each $\om\in\Omega$ as $\delta\to 0$. 
    Hence by dominated convergence, for all sufficiently small $\delta$, $\frac 1n\int f_\delta\,d\rho<\lambda+\ep$
    and $\frac 1n\int g_\delta\,d\rho<-\lambda+\ep$. 
    
    Now let $(\bar\Omega,\bar \rho)$ be an extension of $(\Omega,\rho)$ and let $B$ be a cocycle on $\bar\Omega$
    such that $\|B(\bar\om)-A(\pi(\bar\om))\|\le \delta$ a.e.
    Then by definition $\log \|B^{(n)}(\bar\om)\|\le f_\delta(\pi(\bar\om))$ and
    $\log\|(B^{(n)}(\bar\om)^{-1}\|\le g_\delta(\pi(\bar\om))$.
    It follows that $\lambda_1(B)\le \lambda+\ep$ and $\lambda_m(B)\ge -\frac 1n\int g_\delta(\pi(\bar\om))\,d\rho\ge
    \lambda-\ep$ as required.
\end{proof}

\begin{theorem}[Simplicity]\label{thm:simp}
Let $\sigma$ be an invertible ergodic measure-preserving transformation of a probability space $(\Omega,\rho)$
and let $A\colon\Omega\to\text{Mat}_d(\RR)$ 
%\aq{adopted CGT correction that $\|A\|\le 1$ here and below.}
such that $\|A(\omega)\|\leq 1$ for every $\omega\in\Omega$. Then for all 
$0<\ep<1$, for all $1\le j<d$, and $\rho\times\mathbb{P}$-almost surely,
$$
\lim_{n\to\infty}\frac 1n\log \frac {s_j(A_\ep^{(n)}(\om))}{s_{j+1}(A_\ep^{(n)}(\om))}>0.
$$
That is $\lambda_j^\ep>\lambda_{j+1}^\ep$ for each $j$: the noiselike perturbation of an arbitrary
bounded matrix cocycle has simple Lyapunov spectrum.
\end{theorem}

\begin{proof}
The proof is divided in three steps.

\textsl{Step 0: Initial perturbation}\\
Let $\sigma$ be an ergodic invertible measure-preserving
transformation of $(\Omega,\rho)$ and $A\colon\Omega\to\text{Mat}_d(\RR)$,
such that $\|A(\omega)\|\leq 1$ for every $\omega\in\Omega$.
Let $0<\ep<1$ and let $1\le j\le d-1$ be fixed. We first apply 
Lemma \ref{lem:nonsinginit} to obtain a cocycle $A_0$ where 
$\|A_0(\om)-A(\om)\|\le\frac\ep 2$, $\|A_0(\om)\|\le 1-\frac\ep 2$ 
and $s_d(A_0(\om))\ge \frac\ep 2$ for each $\om\in\Omega$. 

\textsl{Step 1: A nearby extension with simple Lyapunov spectrum}\\
In this step, we find an extension $\Omega'$ of $\Omega$, together with a factor map $\pi:\Omega'\to\Omega$
and a cocycle $A'$ on $\Omega'$ such that $\|A'(\omega')-A(\pi(\omega'))\|<\frac {3\ep}4$ such that
the cocycle $A'$ has simple Lyapunov spectrum.

We inductively build a sequence of extensions of $\Omega$, each one breaking a tie between
a pair of Lyapunov exponents at the previous level, taking care not to create any new ties.
We let $\Omega_0=\Omega$. Then we build an extension $\Omega_1$ of $\Omega_0$ with a new cocycle $A_1$; 
an extension $\Omega_2$ of $\Omega_1$ with a cocycle $A_2$ etc. 

More precisely, we claim the following:
There exist $0\le K\le d-1$, spaces $\Omega_0,\ldots,\Omega_K$; transformations
$\sigma_0,\ldots,\sigma_K$; probability measures $\rho_0,\ldots,\rho_K$; 
factor maps $\pi_1,\ldots,\pi_K$ and cocycles $A_0,\ldots,A_K$ defined 
%\gf{we are overloading subscripts of $A$:  we use variously $\epsilon$, $\omega$, (in previous sections $i$), and now $k$ as single subscripts.  I'm not sure if we  care enough to do something about it. No good solution immediately comes to mind. A small change would be to take instances of $\omega$ subscripts and put them as $A(\omega)$.} \aq{I implemented the $A_\om \to 
%A(\om)$ change (hopefully)}
on $\Omega_0,\ldots,\Omega_K$ such that:
\begin{itemize}
    \item The map $\sigma_k$ is an invertible ergodic measure-preserving transformation of $(\Omega_k,\rho_k)$
    for each $k=0,\ldots,K$;
    \item For each $1\le k\le K$ the map $\pi_k$ is a factor map from $(\Omega_k,\rho_k)$ to $(\Omega_{k-1},\rho_{k-1})$: 
    $\rho_k(\pi_k^{-1}B)=\rho_{k-1}(B)$ for all measurable subsets $B$ of $\Omega_{k-1}$ and
    $\sigma_{k-1}\circ \pi_k=\pi_k\circ\sigma_k$;
    \item Let the Lyapunov exponents of the cocycle $A_k$ over $\Omega_k$ \emph{with repetition} be
    $\mu_{k,1}\ge\ldots\ge \mu_{k,d}$. For any $1\le k\le K$, if $\mu_{k-1,j}>\mu_{k-1,j+1}$, then 
    $\mu_{k,j}>\mu_{k,j+1}$. Further there exists at least one $j$ such that $\mu_{k-1,j}=\mu_{k-1,j+1}$
    while $\mu_{k,j}>\mu_{k,j+1}$;
    \item Each cocycle is a small perturbation (not necessarily noiselike)
    of the previous one: $\|A_k(\om)-A_{k-1}(\pi_k(\om))\|\le\frac\ep{4d}$
    for each $1\le k\le K$;
    \item The $K$th cocycle has simple Lyapunov spectrum: $\mu_{K,1}>\ldots>\mu_{K,d}$.
\end{itemize}

For the base case, $\Omega_0=\Omega$, $\rho_0=\rho$ and we already constructed $A_0$. 
Now suppose $\Omega_0,\ldots,\Omega_k$; $\rho_0,\ldots,\rho_k$; $\sigma_0,\ldots,\sigma_k$ and
$A_0,\ldots,A_k$ have already been constructed to satisfy the above conditions. 

Let the Lyapunov exponents of the $A_k$ cocycle over $\sigma_k$
\emph{without repetition} be $\lambda_{k,1}>\ldots>\lambda_{k,l_k}$
with multiplicities $m_{k,1},\ldots,m_{k,l_k}$.
If the cocycle $A_k$ over $\sigma_k$ has simple Lyapunov spectrum, we let $K=k$ and the induction is complete. 
Otherwise, fix a $j$ such that $m_{k,j}>1$. 
Applying the Multiplicative Ergodic Theorem, 
let $V(\om)$ be the Oseledets space corresponding to the $j$th Lyapunov exponent of the $A_k$ cocycle,
and $U(\om)$ be the direct sum of all of the other Oseledets spaces. For $\PP_k$-a.e.\ $\om$,
$V(\om)$ and $U(\om)$ have trivial intersection and satisfy $A_k(\om)V(\om)=V(\sigma_k\om)$
and $A_k(\om)U(\om)=U(\sigma_k\om)$.
The function $g(\om):=d(U(\om)\cap S,V(\om)\cap S)$
is measurable and almost everywhere positive. Hence there exists an $\eta>0$ such that
$\PP_k(g(\om)>\eta)\ge \frac 23$. We let $G=\{\om\in\Omega_k\colon g(\om)>\eta\}$
and say this is the \emph{good set}. Given $\omega\in \Omega_k$, the $n$'s such that $\sigma_k^n(\om)\in G$
are called the \emph{good times}. 

We modify the cocycle only on the good set. Also our modifications change only what is happening on $V(\om)$, 
leaving all of the cocycle alone on $U(\om)$. That is, we are building an extension $\Omega_{k+1}$
of $\Omega_k$ with factor map $\pi_{k+1}$, and a cocycle on $\Omega_{k+1}$ such that 
such that $A'(\om')|_{U(\om)}=A(\om)|_{U(\om)}$, where $\omega=\pi_{k+1}(\omega')$.
In fact, we will also require $A'(\om')(V(\om))=V(\sigma\om)$
so that the perturbed cocycle preserves the equivariant direct sum $\RR^d=U(\om)\oplus V(\om)$. 
It follows from the Multiplicative Ergodic Theorem that the multiset  that is the 
Lyapunov spectrum (with repetition) of the cocycle $A'$ 
%\aq{OK to include the `with repetition', but to leave out the `$\mu$-'?}\gf{ok}
is the union of the Lyapunov spectra of the restriction of $A'$ to $U(\om)$ and the restriction of $A'$ to $V(\om)$. 
Also, the Lyapunov spectrum of the restriction of $A'$ to $U(\om)$ agrees with the Lyapunov spectrum of the restriction of $A$
to $U(\om)$. 

Notice that $A'(\om')-A(\om)=(A'(\om')-A(\om))\circ \Pi_{V(\om)\parallel U(\om)}$, so that
$\|A'(\om')-A(\om)\|\le \|(A'(\om')-A(\om))|_{V(\om)}\|\,\|\Pi_{V(\om)\parallel U(\om)}\|$. 
By Lemma \ref{lem:projbound}, at good times $\|\Pi_{V(\om)\parallel U(\om)}\|<\frac 2\eta$. 
To ensure that $\|A'(\om')-A(\om)\|\le \frac{\ep}{4d}$, it suffices to ensure that 
$\|(A'(\om)-A(\om))|_{V(\om)}\|\le \frac{\eta\ep}{8d}$. 

We also want to ensure that the Lyapunov exponents of the restriction of the perturbed cocycle to $V(\om)$ remain 
strictly bigger than $\lambda_{j+1}$ and strictly smaller than $\lambda_{j-1}$ (where we take $\lambda_0$ to be
$\infty$ and $\lambda_{l_k+1}$ to be $-\infty$). 
To this end, we use Lemma \ref{lem:semi-cont} to obtain a $\delta$ such that 
if we build an extension $\Omega_{k+1}$ and a cocycle such that 
$|(B_{\bar\om}-A_{\pi(\bar\om)})|_{V(\pi(\bar\om))}|_\infty\le\delta$ for all $\bar\om\in\Omega_{k+1}$,
then the Lyapunov
exponents of the restriction of $B$ to $V(\pi(\bar\om))$ are in the range $(\lambda_{j+1},\lambda_{j-1})$.
Hence this paragraph and the previous paragraph taken together allow us to choose a $\kappa$
so that if the restriction of the extension cocycle to $V(\om)$ differs at good times by at most $\kappa$
and the extension cocycle restricted to $U(\om)$ is unchanged, then the exponents of the restriction of the cocycle
to $V(\om)$ are in the range $(\lambda_{j+1},\lambda_{j-1})$ and $\|A'(\om')-A(\om)\|\le\frac{\ep}{4d}$.

% This means that the norm of the changes to the $A_i$'s are at most $1/\eta$ times the norm of
% the changes as restricted to $V_j$. Since we're doing $d$ inductive steps, and want to end up at most $\frac\ep 2$ away from where 
% we started, we're allowed to make changes of $V_j$-norm at most $\eta\ep/(2d)$ in each good step. 

% We're planning to build $(\Omega_1,A_1)$ by deterministically building targets; and using randomness to fill in the gaps (cf bookkeeping). 

We use a very similar strategy to that in Section \ref{S:2dtarget} where we found a target for $s_1/s_d$.
Let $\gap_k$ be as computed in the start of Theorem \ref{thm:meta} with $\ep$ replaced by $\kappa$
and $d$ replaced by $m_{k,j}$.
Here we aim to create a gap between the fastest and slowest vectors in a block within $V(\om)$.
In a block, either there are two orthogonal vectors whose angle of separation at some stage
within the block is less than $1/\gap_k$; otherwise one may incrementally boost one direction. 

% % restriction of the block to $V(\om)$. 
% % Either two vectors become very close, creating a gap as in Lemma \ref{lem:tard}; or

% % so that you get a win; or you incrementally boost one direction. How close is super-close? 
% Want to gain $\gap$. We should define $\gap$ using $\eta\ep/(8d)$ rather than $\ep$ since we're going to do gluing in a single 
% subspace at a time. Also we're going to do the randomness only at good times. 

Now let $N>\frac{32}{\kappa}\gap_k\log\gap_k$. We apply Rokhlin's lemma to build a Rokhlin tower with base
$H$ of positive $\mathbb{P}_k$-measure contained in the good set $G$
such that the return time to the base always exceeds $N$. Let $r_H$ denote the return time to $H$ under $\sigma_k$.
%\cgt{under $\sigma_k$?}. 
An element $\om\in H$ 
%\cgt{$\om\in \Omega_k$}\aq{not so keen on this because I read it to suggest $\Omega_k$ *is* the base. I could go for $\om\in H$} 
of the base is said to be \emph{good} if $\sum_{i=0}^{r_H(\om)-1}\mathbf 1_G(\sigma_k^i\om)>\frac {r_H(\om)}2$
and let $H'\subset H$ denote the set of good elements of the base. Since $\rho_k(G)\ge \frac 23$,
$H'$ has positive $\rho_k$-measure. For $\om\in H'$, let $B^i(\om)=A_k(\sigma_k^i\om)\cdots A_k(\sigma_k\om)$.
%\cgt{should in previous/subsequent phrase and later $\sigma$ be $\sigma_k$?}\aq{yes -- fixed.}
If there exists a good time $1< i\le r_H(\om)$ such that there exist two orthogonal vectors $v$ and $v'$ in $V(\sigma_k(\om))$
for which $\angle(B^{i-1}(\om) v,B^{i-1}(\om) v')<\frac 1{\gap_k}$, 
then $s_1(B^{i-1}(\om)|_{V(\sigma_k(\om))})/
s_{m_{k,j}}(B^{i-1}(\om)|_{V(\sigma_k(\om))})>\gap_k$ by Lemma \ref{lem:closed2} 
%\gf{Should the index in the ratio of $s$ expression be $i-1$, not $i$?  This would match Lemma 4.3}. 
%\aq{yes to both GF and CGT's comment below here: should be $i-1$ steps starting from sig(om) -- now fixed}.
In that case, the block $B^{i-1}(\om)$ is taken to be the target block. 

Otherwise, for $\omega\in H'$, let $v_1,\ldots,v_{m_{k,j}}$ be a (measurably-chosen) family of singular vectors
for the restriction of $B^{r_H(\om)-1}(\om)$ to $V(\sigma_k(\om))$
with singular values $s_1\ge\ldots\ge s_{m_{k,j}}$.
As in Lemma \ref{lem:far}, we modify the cocycle (but only at the good times), each time expanding the image of $v_1$ by
a factor of $1+\frac\kappa{8\gap_k}$, leaving the other
$v$ images (as well as $U(\sigma_k^i(\om))$) unchanged. 
More explicitly, for $\omega\in H'$ and $n<r_H(\omega)$, we let $P(\sigma_k^n\omega)$ be the projection onto $\lin(B^n(\omega)v_1)$
parallel to $\lin\{B^n(\omega)v_j\colon j>1\}\oplus U(\sigma_k^n(\omega))$; otherwise set $P(\sigma_k^n\omega)=0$.
The perturbation is then defined by
$$
A'(\sigma_k^n(\om))=A(\sigma_k^n(\om))\left(I+\frac \kappa{8\gap_k}\mathbf 1_G(\sigma_k^n(\om))P(\sigma_k^n\omega)\right).
$$

We write ${A'}(\sigma_k(\om)),\ldots,A'(\sigma_k^{r_H(\om)-1}(\om))$ 
for the matrices in the perturbed block and ${B'}^{r_H(\om)-1}(\om)$ for the product
$A'(\sigma_k^{r_H(\om)-1}(\om))\cdots {A'}(\sigma_k(\om))$.
By the end of the block, as in the lemma, since the number of good times exceeds $N/2$,
we have $s_1({B'}^{r_H(\om)-1}(\om))/s_{m_{k,j}}({B'}^{r_H(\om)-1}(\om))>\gap_k$. 
% \cgt{Minor but slightly puzzling: above and below, do blocks like $B_\om^i$ start with $A_k(\sg\om)$ rather than $A_k(\om)$ to have 'transition blocks'? In this case, should arguments of $s_1, s_d$ (just before this) be
% ${A'_{\sg\om}}^{r_H(\om)-1}$?}
% \aq{agreed. Hopefully now fixed.}
%This yields a target block $A'({\sigma(\om)}), \ldots,A'(\sigma^{r_H(\om)-1})$. 
We call this a Type II target.

The coordinates immediately preceding and following a target block are called transition coordinates as in the proof of Theorem
\ref{thm:meta}. 
We then define an extension system $\Omega_{k+1}=\Omega_k\times S_\infty^{\ZZ}$ 
equipped with the measure $\rho_k\times\PP$ where $\PP$ is the i.i.d.\ measure on matrices with
independent uniform $[-1,1]$ elements. We define the cocycle $A_{k+1}$ by
$$
A_{k+1}(\omega,\Xi)=\begin{cases}
    A'(\omega)&\text{if $\omega$ lies inside a Type II target block;}\\
    A(\omega)+\kappa\Xi_0&\text{if $\omega$ is a transition coordinate;}\\
    A(\omega)&\text{otherwise.}
\end{cases}
$$
A calculation exactly analogous to the calculation in Theorem \ref{thm:meta} shows that the restriction of the
$A_{k+1}$ cocycle to the equivariant $V(\om)$ block has non-trivial Lyapunov spectrum. By the choice of 
$\kappa$, all Lyapunov 
exponents lie in $(\lambda_{j-1},\lambda_{j+1})$,  completing this step of the induction. 

\textsl{Step 2: Completion of the proof}\\
At the end of the induction, we have an ergodic invertible extension $\bar\Omega:=\Omega_K$ of 
$\Omega$ and a measurable cocycle $\bar A:=A_K$
with simple Lyapunov spectrum $\lambda_1>\ldots>\lambda_d$. 
Let $\bar\pi=\pi_1\circ\ldots\circ\pi_K$ denote the factor map from 
$\bar\Omega$ to $\Omega$. We have $\|A_0(\bar\pi(\bar\om))-A(\bar\pi(\bar\om))\|<\frac\ep 2$ for each $\bar\om\in\bar\Omega$.
Similarly, by definition of the perturbations, we have
$\|A_k(\pi_{k+1}\circ\ldots\circ\pi_K(\bar\om))-A_{k-1}(\pi_k\circ\ldots\circ\pi_K(\bar\om))\|\le \frac{\ep}{4d}$ for each $k$
and each $\bar\om\in\bar\Omega$. 
Summing these, we obtain for each $\bar\om\in\bar\Omega$,
\begin{align*}
\|\bar A(\bar\om)-A(\bar\pi(\bar\om))\|&\le
\|\bar A(\bar\om)-A_0(\bar\pi(\bar\om))\|+\|A_0(\bar\pi(\bar\om))-A(\bar\pi(\bar\om))\|
<\tfrac{3\ep}4.
\end{align*}
Let $\gap$ be the required gap identified in the beginning of the proof of Theorem \ref{thm:meta}
for perturbations of size $\ep$ and dimension $d$. 
From Raghunathan's proof of Oseledets' theorem \cite{Raghunathan}, we know that 
$\lim_{n\to\infty}\frac 1n\log s_j(\bar A^{(n)}(\bar\om))\to\lambda_j$ for each $1\le j\le d$ and for a.e.\ $\bar\om\in\bar\Omega$. 
Since the Lyapunov spectrum is simple, it follows that $s_j(\bar A^{(n)}(\bar\om))/s_{j+1}(\bar A^{(n)}(\bar\om))\to\infty$
for a.e.\ $\bar\om\in\bar\Omega$.
Hence, there exists an $N$ such that the probability that 
$s_j(\bar A^{(N-1)}(\bar\om))/s_{j+1}(\bar A^{(N-1)}(\bar\om))$ exceeds $\gap$
is at least $\frac 12$. Let $\bar G\subset \bar \Omega$ be this set. 

Since $\bar\rho$ is an ergodic invariant measure, $\bar\rho$-a.e.\ $\bar\om$ hits $\bar G$ 
with frequency equal to $\bar\rho(\bar G)\ge \frac 12$. 
Since $\bar\rho$ may not be ergodic with respect to $\sigma^{N}$, we are unable to conclude 
that for a.e.\ $\bar\om\in\bar\Omega$, 
$\lim_{M\to\infty}\#\{n\le M\colon \bar\sigma^{nN}\bar\om\in\bar G\}/M\ge\rho(\bar G)$,
however this inequality holds on a set of measure at least $\frac 1N$. Call this set $\bar H$. 

Let $\bar\omega\in\bar H$ and consider a realization $(X_n)_{n\in\ZZ}$ of the cocycle $A_\ep$
where $X_n=A_{\sigma^n(\bar\pi(\bar\omega))}+ \ep\Xi_n$.
We say that the realization hits the target area on the $\ell$th
%\gf{I changed $k$ to $\ell$ from here to the end of the proof just to get away from $k$ indexing.} 
block if the following two conditions are satisfied:
\begin{enumerate}
    \item 
$s_j(\bar A^{(N-1)}(\bar\sigma^{\ell N}\bar\om))/s_{j+1}(\bar A^{(N-1)}(\bar\sigma^{\ell N}\bar\om))>\gap$; and 
\item 
$|X_n-\bar A_{\bar\sigma^n(\bar\omega)}|_\infty\le (\frac\ep4)^N/(3dN)$ for $n=\ell N,\ldots,\ell N+(N-1)$.
%\cgt{I should still figure why the '3'...}\aq{This was a mistake. it should be $\ell N$; not $3\ell N$ (now fixed).
%There were just 3's floating around. Should be $\ell N$}
\end{enumerate}
By definition of $\bar H$, the first condition is satisfied with frequency at least $\frac 12$.
Given that the first condition is satisfied, the block $(A_\ep(\sigma^n\om)+\ep\Xi_n)_{n=\ell N}^{\ell N+(N-2)}$
hits the target area with probability $[(\frac\ep 4)^N/(3dN)]^{d^2(N-1)}$ (independently of whether 
other targets are hit). 
Following the proof of Theorem \ref{thm:meta}, we see that the $j$th exponent of the $(A_{n,\ep}(\om))$ cocycle is
strictly larger than the $(j+1)$st. 
\end{proof}

\section{Proofs of Main Theorems}
\begin{proof}[Proof of Theorem \ref{mt:dxd}]
Lemma \ref{lem:tard} shows that the hypothesis of Theorem \ref{thm:meta} is satisfied. 
From the proof of Theorem \ref{thm:meta}, $\gap=C/\ep^{8d}$. From Lemma \ref{lem:tard},
the block length is given by $N=(16/\ep)\gap\log\gap$, so that $N=\ep^{-(8d+1+o(1))}$. 
The lower bound is then given by the quantity $\frac pN$, where $p$ appears in Corollary \ref{cor pos measure}.
That is, $c_d'(\ep)=(\frac\ep 4)^{d^2N^2}/[N (3dN)^{d^2N}]$, so that
$c_d'(\ep)=\exp\big(-4\ep^{-(16d+2+o(1))}|\log\frac\ep 4|-4\ep^{-(8d+1+o(1))}\log (3dN)-\log N\big)$.
Hence $c_d'(\ep)=\exp(-\ep^{-(16d+2+o(1)}))$ as claimed.
\end{proof}

Note that in the case $d=2$, $s_1(A)/s_d(A)=s_1(A)/s_2(A)$. Hence $c_2(\ep)=c_2'(\ep)$, so that
Theorem \ref{mt:2x2} is a special case of Theorem \ref{mt:dxd}.

% \begin{proof}[Proof of Theorem \ref{mt:2x2}]
% Lemma \ref{lem:tard} shows that the hypothesis of Theorem \ref{thm:meta} is satisfied. 
% From the proof of Theorem \ref{thm:meta}, $\gap=C/\ep^{16}$. From Lemma \ref{lem:tard},
% the block length is given by $N=(16/\ep)\gap\log\gap$, so that $N=\ep^{-(17+o(1))}$. 
% The lower bound is then given by the quantity $\frac pN$, where $p$ appears in Corollary \ref{cor pos measure}.
% That is, $c_2(\ep)=(\frac\ep 4)^{d^2N^2}/[N (3dN)^{d^2N}]$, so that
% $c_2(\ep)=\exp\big(-4\ep^{-(34+o(1))}|\log\frac\ep 4|-4\ep^{-17+o(1)}\log (3dN)-\log N\big)$.
% Hence $c_2(\ep)=\exp(-\ep^{-(34+o(1)}))$ as claimed.
% \end{proof}

% We remark that the same proof establishes Theorem \ref{mt:dxd} (where in the explicit bound
% $\gap=C/\ep^{8d}$, $N=\ep^{-(8d+1+o(1))}$).

\begin{proof}[Proof of Theorem \ref{mt:3x3}]
Lemma \ref{lem:tar3} in the case $j=1$ or Corollary \ref{cor:tar3} in the case $j=2$
shows that the hypothesis of Theorem \ref{thm:meta} is satisfied.
From the proof of Theorem \ref{thm:meta}, we take $\gap=C/\ep^{8d}=C/\ep^{24}$. From Lemma \ref{lem:tar3}
or Corollary \ref{cor:tar3},
for small $\ep$, $\eta=\gap^{-9}=C\ep^{216}$, so that $N=\ep^{-433+o(1)}$, giving 
$c_3(\ep)=\exp(-1/\ep^{866+o(1)})$ as claimed.
\end{proof}

% \begin{proof}[Proof of Theorem \ref{mt:dxd}]
% The proof is as above using Corollary \ref{cor:tard} and Theorem \ref{thm:meta}.
% \end{proof}

% We comment that the obstruction to proving a full quantitative Lyapunov spectrum
% simplicity result in dimensions $d>3$ 
% lies in the identification of
% %is that we do not know how to build 
% a target separating 
% the $j$th and $(j+1)$st singular values. 
% Once a procedure to build targets is found, Theorem \ref{thm:meta} will give rise to the required 
% lower bounds on the corresponding gaps.

\begin{proof}[Proof of Theorem \ref{mt:ort+unif}]
Let $1\le j<d$. Let $\gap=C/\ep^{8d}$ as required in Theorem \ref{thm:meta}.
Since orthogonal matrices preserve orthogonal frames, any block is $\sqrt 2$-spread. 
Letting $N$ be as in the statement of Lemma \ref{lem:far}, we see $N=\ep^{-(1+o(1))}$.
Substituting in the expression for $\frac pN$ in Theorem \ref{thm:meta}, we see 
that $c(\ep)=\exp(-1/\ep^{2+o(1)})$ as required.
%\gf{Do we mean to say that ``the condition in Theorem \ref{thm:meta} is satisfied.''? Don't we instead compute the order of $p/N$?}\aq{I agree. Hopefully it's better now?}
\end{proof}

\begin{proof}[Proof of Theorem \ref{mt:dxd2}]
Let $\ep>0$. Let $\gap=C/\ep^{8d}$ be as constructed at the start of the proof of Theorem \ref{thm:meta}.
Suppose for a contradiction that there does not exist a $c_d(\ep)>0$ such that for all sequences 
of matrices $(A_n)$ of norm at most 1 and all $1\le j\le d-1$, 
one has $\liminf_{n\to\infty} \frac 1n \log[s_j(A_\ep^{(n)})/s_{j+1}(A_\ep^{(n)})]\ge c_d(\ep)$. 
That is, there exist $1\le j\le d-1$ and sequences of matrices $(A_n)$ such that
$\liminf_{n\to\infty} \frac 1n \log[s_j(A_\ep^{(n)})/s_{j+1}(A_\ep^{(n)})]$ is arbitrarily close to 0.
Then the hypothesis of Lemma \ref{lem:compact} is satisfied. So one concludes that
there exists an ergodic invariant measure on $S^\ZZ$ for which $\lambda_j^{\ep/(4d)}=
\lambda_{j+1}^{\ep/(4d)}$. However, this contradicts the 
conclusion of Theorem \ref{thm:simp}, so that there must exist a positive universal gap $c_d(\ep)>0$.
\end{proof}

\bibliographystyle{abbrv}
\bibliography{AFGTQ}

@article{Raghunathan,
    author = {M.~S. Raghunathan},
     TITLE = {A proof of {O}seledec's multiplicative ergodic theorem},
   JOURNAL = {Israel J. Math.},
    VOLUME = {32},
      YEAR = {1979},
     PAGES = {356--362},
}

@article{Bochi-Gourmelon,
    author = {J. Bochi and N. Gourmelon},
    title = {Some characterizations of domination},
    journal = {Math. Z.},
    volume = {209},
    year = {2009},
    pages = {221--231}
}

@book {DuarteKlein,
    AUTHOR = {Duarte, Pedro and Klein, Silvius},
     TITLE = {Lyapunov exponents of linear cocycles},
      NOTE = {Continuity via large deviations},
 PUBLISHER = {Atlantis Press, Paris},
      YEAR = {2016}
}

@book{FurstenbergBook,
    author = {H. Furstenberg},
     TITLE = {Recurrence in ergodic theory and combinatorial number theory},
 PUBLISHER = {Princeton University Press, Princeton, NJ},
      YEAR = {1981}
}

@article{FGTQ,
	title = {Stochastic {Stability} of {Lyapunov} {Exponents} and {Oseledets} {Splittings} for {Semi}-invertible {Matrix} {Cocycles}},
	volume = {68},
	journal = {Communications on Pure and Applied Mathematics},
	author = {Froyland, Gary and González-Tokman, Cecilia and Quas, Anthony},
	year = {2015},
	pages = {2052--2081},
}

@article {FGTQ19,
    AUTHOR = {Froyland, Gary and Gonz\'{a}lez-Tokman, Cecilia and Quas,
              Anthony},
     TITLE = {Hilbert space {L}yapunov exponent stability},
   JOURNAL = {Trans. Amer. Math. Soc.},
    VOLUME = {372},
      YEAR = {2019},
     PAGES = {2357--2388},
}

@article {Young86,
    AUTHOR = {Young, L.-S.},
     TITLE = {Random perturbations of matrix cocycles},
   JOURNAL = {Ergodic Theory Dynam. Systems},
    VOLUME = {6},
      YEAR = {1986},
     PAGES = {627--637},
}

@article {LedrappierYoung91,
    AUTHOR = {Ledrappier, F. and Young, L.-S.},
     TITLE = {Stability of {L}yapunov exponents},
   JOURNAL = {Ergodic Theory Dynam. Systems},
    VOLUME = {11},
      YEAR = {1991},
     PAGES = {469--484},
}

@article {ImkellerLederer,
    AUTHOR = {Imkeller, Peter and Lederer, Christian},
     TITLE = {An explicit description of the {L}yapunov exponents of the
              noisy damped harmonic oscillator},
   JOURNAL = {Dynam. Stability Systems},
    VOLUME = {14},
      YEAR = {1999},
     PAGES = {385--405},  
}

@article {BaxendaleGoukasian,
    AUTHOR = {Baxendale, Peter H. and Goukasian, Levon},
     TITLE = {Lyapunov exponents for small random perturbations of
              {H}amiltonian systems},
   JOURNAL = {Ann. Probab.},
    VOLUME = {30},
      YEAR = {2002},
     PAGES = {101--134},
}

@article {BlumenthalYun,
    AUTHOR = {Blumenthal, Alex and Yang, Yun},
     TITLE = {Positive {L}yapunov exponent for random perturbations of
              predominantly expanding multimodal circle maps},
   JOURNAL = {Ann. Inst. H. Poincar\'{e} C Anal. Non Lin\'{e}aire},
    VOLUME = {39},
      YEAR = {2022},
     PAGES = {419--455},
}

@article {ChemnitzEngel,
    AUTHOR = {Chemnitz, Dennis and Engel, Maximilian},
     TITLE = {Positive {L}yapunov exponent in the {H}opf normal form with
              additive noise},
   JOURNAL = {Comm. Math. Phys.},
    VOLUME = {402},
      YEAR = {2023},
     PAGES = {1807--1843},
}

@article {Bochi,
    AUTHOR = {Bochi, Jairo},
     TITLE = {Genericity of zero {L}yapunov exponents},
   JOURNAL = {Ergodic Theory Dynam. Systems},
    VOLUME = {22},
      YEAR = {2002},
     PAGES = {1667--1696},
}

@article {BochiViana,
    AUTHOR = {Bochi, Jairo and Viana, Marcelo},
     TITLE = {The {L}yapunov exponents of generic volume-preserving and
              symplectic maps},
   JOURNAL = {Ann. of Math. (2)},
    VOLUME = {161},
      YEAR = {2005},
     PAGES = {1423--1485},
}

@article {VianaContReview,
    AUTHOR = {Viana, Marcelo},
     TITLE = {({D}is)continuity of {L}yapunov exponents},
   JOURNAL = {Ergodic Theory Dynam. Systems},
    VOLUME = {40},
      YEAR = {2020},
     PAGES = {577--611},
}

@book {HH,
    AUTHOR = {Hall, P. and Heyde, C. C.},
     TITLE = {Martingale limit theory and its application.},
 PUBLISHER = {Academic Press, Inc. [Harcourt Brace Jovanovich, Publishers],
              New York-London},
      YEAR = {1980},
     PAGES = {xii+308},
}

@article {BenedicksCarleson,
    AUTHOR = {Benedicks, Michael and Carleson, Lennart},
     TITLE = {The dynamics of the {H}\'{e}non map},
   JOURNAL = {Ann. of Math. (2)},
    VOLUME = {133},
      YEAR = {1991},
     PAGES = {73--169},
}

@article {BlumenthalXueYoung17,
    AUTHOR = {Blumenthal, Alex and Xue, Jinxin and Young, Lai-Sang},
     TITLE = {Lyapunov exponents for random perturbations of some
              area-preserving maps including the standard map},
   JOURNAL = {Ann. of Math. (2)},
    VOLUME = {185},
      YEAR = {2017},
     PAGES = {285--310},
}

@article {BlumenthalXueYoung18,
    AUTHOR = {Blumenthal, Alex and Xue, Jinxin and Young, Lai-Sang},
     TITLE = {Lyapunov exponents and correlation decay for random
              perturbations of some prototypical 2{D} maps},
   JOURNAL = {Comm. Math. Phys.},
    VOLUME = {359},
      YEAR = {2018},
     PAGES = {347--373},
}

@book {KiferBook,
    AUTHOR = {Kifer, Yuri},
     TITLE = {Random perturbations of dynamical systems},
    SERIES = {Progress in Probability and Statistics},
    VOLUME = {16},
 PUBLISHER = {Birkh\"{a}user Boston, Inc., Boston, MA},
      YEAR = {1988},
     PAGES = {vi+294},
}

@article {Khasminskii,
    AUTHOR = {Khasminski\u{\i}, R. Z.},
     TITLE = {The averaging principle for parabolic and elliptic
              differential equations and {M}arkov processes with small
              diffusion},
   JOURNAL = {Teor. Verojatnost. i Primenen.},
    VOLUME = {8},
      YEAR = {1963},
     PAGES = {3--25},
}

@article {Young13,
    AUTHOR = {Young, Lai-Sang},
     TITLE = {Mathematical theory of {L}yapunov exponents},
   JOURNAL = {J. Phys. A},
    VOLUME = {46},
      YEAR = {2013},
     PAGES = {254001, 17},
}

@article {Young08,
    AUTHOR = {Young, Lai-Sang},
     TITLE = {Chaotic phenomena in three settings: large, noisy and out of
              equilibrium},
   JOURNAL = {Nonlinearity},
    VOLUME = {21},
      YEAR = {2008},
     PAGES = {T245--T252},
}

@article {LianStenlund,
    AUTHOR = {Lian, Zeng and Stenlund, Mikko},
     TITLE = {Positive {L}yapunov exponent by a random perturbation},
   JOURNAL = {Dyn. Syst.},
    VOLUME = {27},
      YEAR = {2012},
     PAGES = {239--252},
}

@article {CowiesonYoung,
    AUTHOR = {Cowieson, William and Young, Lai-Sang},
     TITLE = {S{RB} measures as zero-noise limits},
   JOURNAL = {Ergodic Theory Dynam. Systems},
    VOLUME = {25},
      YEAR = {2005},
     PAGES = {1115--1138},
}

@article{AvilaViana07,
author = {Artur Avila and Marcelo Viana},
title = {{Simplicity of Lyapunov spectra: proof of the Zorich-Kontsevich conjecture}},
volume = {198},
journal = {Acta Mathematica},
publisher = {Institut Mittag-Leffler},
pages = {1 -- 56},
year = {2007},
}

@book {VianaBook,
    AUTHOR = {Viana, Marcelo},
     TITLE = {Lectures on {L}yapunov exponents},
    SERIES = {Cambridge Studies in Advanced Mathematics},
    VOLUME = {145},
 PUBLISHER = {Cambridge University Press, Cambridge},
      YEAR = {2014},
     PAGES = {xiv+202},
}

@article {Lyapunov,
    AUTHOR = {Lyapunov, A. M.},
     TITLE = {The general problem of the stability of motion},
      NOTE = {Translated by A. T. Fuller from \'{E}douard Davaux's French
              translation (1907) of the 1892 Russian original,
              With an editorial (historical introduction) by Fuller, a
              biography of Lyapunov by V. I. Smirnov, and the bibliography
              of Lyapunov's works collected by J. F. Barrett,
              Lyapunov centenary issue},
   JOURNAL = {Internat. J. Control},
    VOLUME = {55},
      YEAR = {1992},
     PAGES = {521--790},
}

@article {Oseledets,
    AUTHOR = {Oseledec, V. I.},
     TITLE = {A multiplicative ergodic theorem. {C}haracteristic {L}japunov,
              exponents of dynamical systems},
   JOURNAL = {Trudy Moskov. Mat. Ob\v{s}\v{c}.},
    VOLUME = {19},
      YEAR = {1968},
     PAGES = {179--210},
}

@article {FurstenbergKesten,
    AUTHOR = {Furstenberg, H. and Kesten, H.},
     TITLE = {Products of random matrices},
   JOURNAL = {Ann. Math. Statist.},
    VOLUME = {31},
      YEAR = {1960},
     PAGES = {457--469},
}

@article {Furstenberg63,
    AUTHOR = {Furstenberg, Harry},
     TITLE = {Noncommuting random products},
   JOURNAL = {Trans. Amer. Math. Soc.},
    VOLUME = {108},
      YEAR = {1963},
     PAGES = {377--428},
}

@incollection {GuivarchRaugi,
    AUTHOR = {Guivarc'h, Y. and Raugi, A.},
     TITLE = {Products of random matrices: convergence theorems},
 BOOKTITLE = {Random matrices and their applications ({B}runswick, {M}aine,
              1984)},
    SERIES = {Contemp. Math.},
    VOLUME = {50},
     PAGES = {31--54},
 PUBLISHER = {Amer. Math. Soc., Providence, RI},
      YEAR = {1986},
}

@article {GoldsheidMargulis,
    AUTHOR = {Gol\'dshe\u{\i}d, I. Ya. and Margulis, G. A.},
     TITLE = {Lyapunov exponents of a product of random matrices},
   JOURNAL = {Uspekhi Mat. Nauk},
    VOLUME = {44},
      YEAR = {1989},
     PAGES = {13--60},
}

@article {ArnoldCong97,
    AUTHOR = {Arnold, Ludwig and Cong, Nguyen Dinh},
     TITLE = {On the simplicity of the {L}yapunov spectrum of products of
              random matrices},
   JOURNAL = {Ergodic Theory Dynam. Systems},
    VOLUME = {17},
      YEAR = {1997},
     PAGES = {1005--1025},
}

@article {ArnoldCong99,
    AUTHOR = {Arnold, Ludwig and Cong, Nguyen Dinh},
     TITLE = {Linear cocycles with simple {L}yapunov spectrum are dense in
              {$L^\infty$}},
   JOURNAL = {Ergodic Theory Dynam. Systems},
    VOLUME = {19},
      YEAR = {1999},
     PAGES = {1389--1404},
}

@article {BonattiViana04,
    AUTHOR = {Bonatti, C. and Viana, M.},
     TITLE = {Lyapunov exponents with multiplicity 1 for deterministic
              products of matrices},
   JOURNAL = {Ergodic Theory Dynam. Systems},
    VOLUME = {24},
      YEAR = {2004},
     PAGES = {1295--1330},
}

@article {BPVL-sl,
    AUTHOR = {Backes, Lucas and Poletti, Mauricio and Varandas, Paulo and
              Lima, Yuri},
     TITLE = {Simplicity of {L}yapunov spectrum for linear cocycles over
              non-uniformly hyperbolic systems},
      NOTE = {With an appendix by Lima},
   JOURNAL = {Ergodic Theory Dynam. Systems},
    VOLUME = {40},
      YEAR = {2020},
     PAGES = {2947--2969},
}

@article {PolettiViana-sl,
    AUTHOR = {Poletti, Mauricio and Viana, Marcelo},
     TITLE = {Simple {L}yapunov spectrum for certain linear cocycles over
              partially hyperbolic maps},
   JOURNAL = {Nonlinearity},
    VOLUME = {32},
      YEAR = {2019},
     PAGES = {238--284},
}

@article {MatheusMollerYoccoz-sl,
    AUTHOR = {Matheus, Carlos and M\"{o}ller, Martin and Yoccoz,
              Jean-Christophe},
     TITLE = {A criterion for the simplicity of the {L}yapunov spectrum of
              square-tiled surfaces},
   JOURNAL = {Invent. Math.},
    VOLUME = {202},
      YEAR = {2015},
     PAGES = {333--425},
}

@article {AvilaViana07-1,
    AUTHOR = {Avila, Artur and Viana, Marcelo},
     TITLE = {Simplicity of {L}yapunov spectra: a sufficient criterion},
   JOURNAL = {Port. Math. (N.S.)},
    VOLUME = {64},
      YEAR = {2007},
     PAGES = {311--376},
}

@misc{BednarskiQuas,
      title={Lyapunov exponents of orthogonal-plus-normal cocycles}, 
      author={Sam Bednarski and Anthony Quas},
      year={2023},
      eprint={2309.08193},
      archivePrefix={arXiv},
      primaryClass={math.DS},
      note={ArXiv 2309.08193}
}

@misc{GorodetskiKleptsyn,
      title={Non-stationary version of Ergodic Theorem for random dynamical systems}, 
      author={Anton Gorodetski and Victor Kleptsyn},
      year={2023},
      eprint={2305.05028},
      archivePrefix={arXiv},
      primaryClass={math.DS},
note={ArXiv 2305.05028}
}

@article {GoldsteinSchlag,
    AUTHOR = {Goldstein, Michael and Schlag, Wilhelm},
     TITLE = {H\"{o}lder continuity of the integrated density of states for
              quasi-periodic {S}chr\"{o}dinger equations and averages of
              shifts of subharmonic functions},
   JOURNAL = {Ann. of Math. (2)},
    VOLUME = {154},
      YEAR = {2001},
     PAGES = {155--203},
}

\end{document}